\documentclass[amsmath,secnumarabic,amssymb,floatfix,nofootinbib,nobibnotes,letterpaper,11pt,tightenlines,showkeys]{revtex4}

\usepackage{times}

\usepackage{geometry}
\usepackage{amssymb}
\usepackage{latexsym, amsmath, amscd,amsthm}
\usepackage{graphicx}
\usepackage{array}
\usepackage[percent]{overpic}
\usepackage{units}
\usepackage{hyperref}

\usepackage{algorithm}
\usepackage{algpseudocode}
\usepackage{siunitx}

\floatstyle{plain}
\newfloat{algfloat}{tb}{fan}
\floatname{algfloat}{Algorithm}

\newtheorem{theorem}{Theorem}

\newtheorem{lemma}[theorem]{Lemma}
\newtheorem{proposition}[theorem]{Proposition}

\theoremstyle{definition}
\newtheorem{definition}[theorem]{Definition}

\newcommand{\Z}{\mathbb{Z}}

\newcommand{\loopinsert}{E_1}
\newcommand{\edgedouble}{E_2}
\newcommand{\cutedgedouble}{E_3}
\newcommand{\pairinsert}{E_4}
\newcommand{\plantri}{\textit{plantri} }

\newcommand{\valgrind}{\textit{valgrind} }

\newcommand{\pdcode}{pd-code }
\newcommand{\pdcodes}{pd-codes }

\newcommand{\edges}{\operatorname{edges}}
\newcommand{\head}{\operatorname{head}}
\newcommand{\tail}{\operatorname{tail}}
\newcommand{\Cr}{\operatorname{Cr}}
\newcommand{\tref}{3_1}

\newcommand{\twistunknot}{\raisebox{-0.17\baselineskip}{\includegraphics[height=0.81\baselineskip]{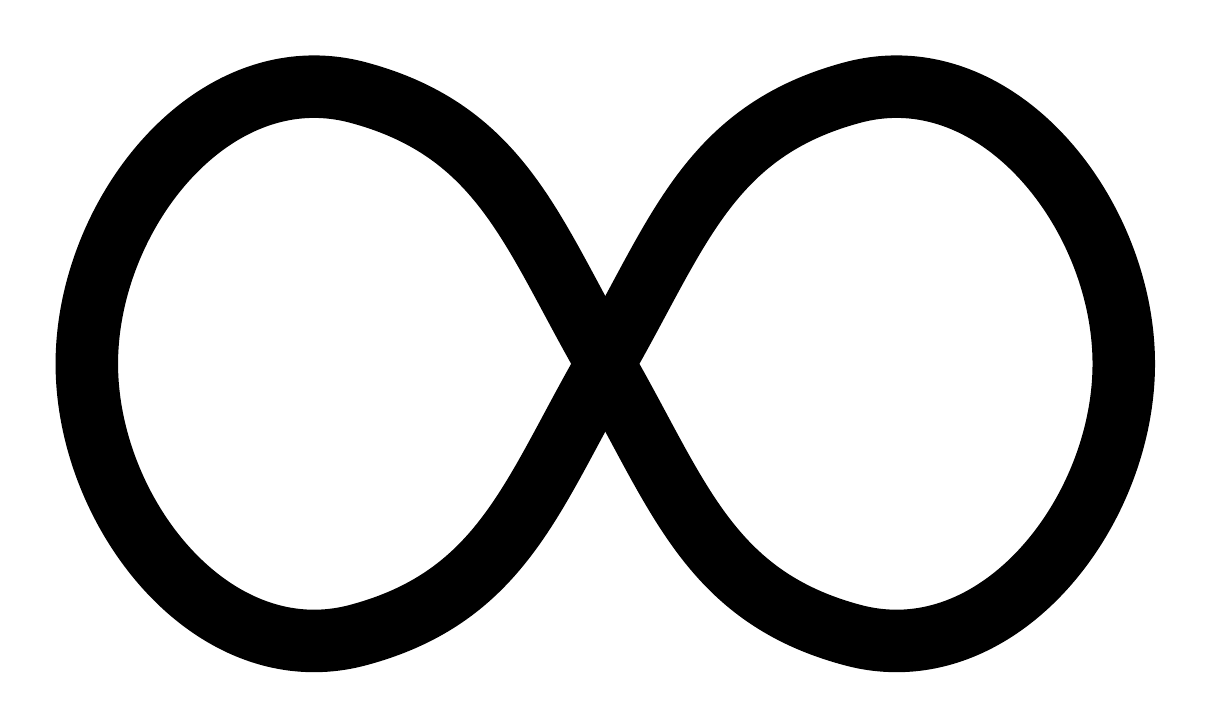}}}
\newcommand{\hopfgraph}{\raisebox{-0.17\baselineskip}{\includegraphics[height=0.81\baselineskip]{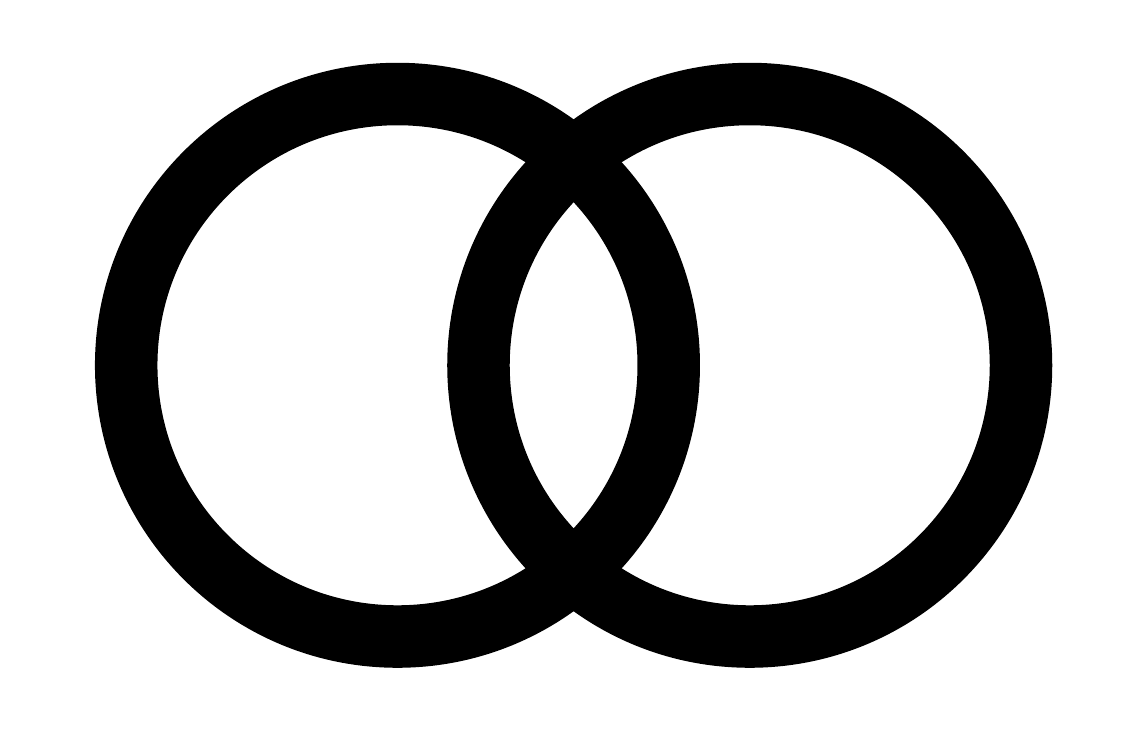}}}

\let\mgp=\marginpar \marginparwidth18mm \marginparsep1mm
\def\marginpar#1{\mgp{\raggedright\tiny #1}}

\let\lbl=\label
\def\label#1{\lbl{#1}\ifinner\else\marginpar{\ref{#1} #1}\ignorespaces\fi}

\begin{document}
\title[]{Knot Probabilities in Random Diagrams}
\author{Jason Cantarella, Harrison Chapman}
\affiliation{University of Georgia, Mathematics Department, Athens GA}
\author{Matt Mastin}
\affiliation{MailChimp, Atlanta, GA}

\begin{abstract}
We consider a natural model of random knotting-- choose a knot diagram at random from the finite set of diagrams with $n$ crossings. We tabulate diagrams with 10 and fewer crossings and classify the diagrams by knot type, allowing us to compute exact probabilities for knots in this model. As expected, most diagrams with 10 and fewer crossings are unknots (about $78\%$ of the roughly 1.6 billion 10 crossing diagrams). For these crossing numbers, the unknot fraction is mostly explained by the prevalence of ``tree-like'' diagrams which are unknots for any assignment of over/under information at crossings. The data shows a roughly linear relationship between the log of knot type probability and the log of the frequency rank of the knot type, analogous to Zipf's law for word frequency. The complete tabulation and all knot frequencies are included as supplementary data.
\end{abstract}


\maketitle

\section{Introduction}

The study of random knots goes back to the 1960's, when polymer physicists realized that the knot type of a closed (or ring) polymer would play an important role in the statistical mechanics of the polymer~\cite{Edwards:1967hd}. Orlandini and Whittington~\cite{Orlandini:2007kn} give a comprehensive survey of the development of the field in the subsequent years. Most random knot models are based on closed random walks, and there are many variants corresponding to different types of walks (lattice walks, random walks with fixed edgelengths, random walks with variable edgelengths, random walks with different types of geometric constraints such as fixed turning angles). For almost all of these models, certain phenomena have been rigorously established-- for instance, the probability of knotting goes to 1 exponentially fast as the size of the walk increases\cite{Pippenger:1989ft,Sumners:1999cd,Diao:2001db}. However, it has been difficult to prove more informative theorems.

One way to look at the problem is these knot models are all models of random space curves, and it is quite hard to relate the three-dimensional shape of a space curve to its knot type. One can in principle express the finite-type invariants as (complicated) integrals in the spirit of \cite{Lin:1994wq}, but so far computing the expected value of these integrals has been too difficult.

The point of this paper is to look at knot diagrams not as convenient combinatorial representations of space curves, but as a probability space in their own right. This is basically the same model of random knotting as \cite{Dunfield:mdWrGjny} or \cite{Diao:2005tp}: the objects are equivalence classes of immersions of $S^1$ into $S^2$ (as in Arnol'd~\cite{Arnold:1994wr}) paired with assignments of orientation and crossing signs. There are other quite different combinatorial models of random knotting in the literature-- see \cite{Cohen:2015wz}, \cite{EvenZohar:2014ws}, or \cite{Nechaev:1996gv}.

Previous authors~\cite{Diao:2005tp} have sampled a slightly different version of this space by using an algorithm of Schaeffer and Zinn-Justin~\cite{Schaeffer:2004tt} to uniformly sample a space of marked link diagrams and keeping only the knots. We wanted to get exact values for probabilities and also to study the probability of very rare knots, so a sampling approach was not appropriate for our purposes. Instead, we carried out a complete enumeration of knot diagrams up to 10 crossings using the graph theory software~\emph{plantri}~\cite{Brinkmann:2005um,Brinkmann:2007up}.

This paper is primarily computational. While we give a proof that our algorithm for enumerating diagrams is correct, our main contribution is the data set of diagrams itself. To generate the diagrams, we first enumerated the unoriented immersions of $S^1$ into $S^2$ in two different ways, as described below. We call these immersions knot shadows\footnote{Kauffman~\cite{Kauffman:1988td} calls these knot \emph{universes}.}.  The sets of shadows generated in each way were identical, and we checked the number of shadows against counts by Arnol'd for $n \leq 5$~\cite[page 79]{Arnold:1994wr}, sequence A008989 in the Online Encyclopedia of Integer Sequences~\cite{oeis}, and later enumerations of Kapoln\'ai et al.~\cite{Kapolnai:2012hs} and Coqueraux et al.~\cite{Coquereaux:2015wv}. We then enumerate assignments of crossing information and orientation for each shadow up to diffeomorphisms of $S^2$ preserving the orientation of the curve. (This is more complicated than it first seems; if the diagram has a symmetry, not all assignments of crossing information are different.) Finally, we compute the knot type for each assignment of crossing and orientation. These computations are new and fairly substantial. Enumerating the 1.6 billion 10-crossing knot diagrams and computing their HOMFLY-PT polynomial took several thousand hours of CPU time on the Amazon EC2 cloud computing service.

Figure~\ref{fig:knot frequency loglog} shows the relative frequency of all the knot types we observed, sorted by rank order among knots. When plotted on a log-log plot, we see that the plot is roughly linear across 9 orders of magnitude-- giving some evidence for a roughly Zipfian distribution of knot types. In such a distribution, the $k$-th most frequent knot type would have a probability proportional to $k^{-s}$ for some $s$. It would be very interesting to know what happens for large $n$.

How many of these diagrams represent nontrivial knots? Relatively few. Even for 10-crossing diagrams, the proportion of unknots is about $77\%$. For these crossing numbers, this is largely explained by the surprising frequency of ``tree-like'' knot shadows, studied by Aicardi~\cite{Aicardi:1994uq}, (Figure~\ref{fig:proportiongraphic}) for which the knot type does not depend on the assignment of crossing information-- the resulting knot is \emph{always} the unknot. These diagrams are surprisingly common: $42.05\%$ of 8-crossing diagrams are treelike, which explains about half of the unknots among 8-crossing diagrams ($84\%$ of 8-crossing diagrams are unknots). The remaining unknots are almost all connect sums of treelike diagrams with the unique prime 3- or 4-crossing shadow. We will show that a simple analysis (Proposition~\ref{prop:almost treelike mostly unknots}) based on such diagrams gives a lower bound of $77\%$ for the unknot fraction in 8-crossing diagrams, explaining more than $90\%$ of all unknots in this class.

Tree-like diagrams are composite diagrams where each prime summand is one-crossing knot diagram $\twistunknot$. In fact, random diagrams are in general highly composite~\cite{Chapman2015knotasymp} (Figure~\ref{fig:proportiongraphic}) with many simple summands. It remains an interesting open question to try to characterize the asymptotic distribution of sizes and numbers of connect summands in a random diagram for large $n$.

In general, counting knot diagrams is made significantly more complicated by the existence of diagram symmetries; for instance, a correction factor due to symmetries is the main technical difference between this natural model and the model of~\cite{Schaeffer:2004tt}. Our data shows that this difficulty rapidly disappears as the number of crossings increases; the average size of the automorphism group for a random 10 crossing knot diagram, for instance, is $846929/823832 \simeq 1.028$. This means that we expect this model to behave very much like the ``rooted'' model of~\cite{Schaeffer:2004tt} for crossing numbers 11 and above; one of us (Chapman) has shown this is true for sufficiently many crossings~\cite{Chapman2015knotasymp}.

\section{Definitions}

We begin with some definitions.
\begin{definition}
We define a \emph{link shadow} $L$ with $n$ vertices to be an equivalence class of generic smooth immersions of a collection of oriented circles into $S^2$ with $n$ intersections up to diffeomorphism of the sphere (which may not be orientation-preserving on either the sphere or the circles; Arnol'd~\cite{Arnold:1994wr} calls this equivalence ``OU'').
\end{definition}
Examples of link shadows are shown in Figure \ref{fig:ShadowExamples}.
\begin{figure}[htp]
	\begin{overpic}[width=4in]{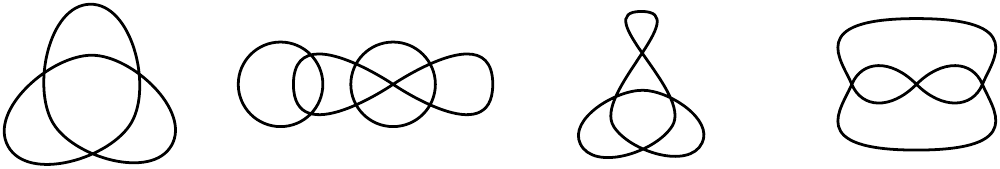}
	\end{overpic}
	\caption{\label{fig:ShadowExamples}Examples of link shadows.}
\end{figure}
In our experiments, each link shadow will be represented (non-uniquely) by a combinatorial object called a PD-code.
\begin{definition}\label{def:PD}
A \textbf{pd-code} $\mathfrak{L}$ with $n$-vertices is a list of $n$ cyclically ordered quadruples called vertices which partition the set of signed edges $\{\pm 1,\dots, \pm 2n\}$ so that in each quadruple every pair of non-adjacent edges have opposite signs.

Two pd-codes are \textbf{pd-isomorphic} if there exists a bijection of vertices and edges which respects the partition of signed edges into vertices, including globally preserving or reversing the cyclic ordering of edges around vertices. A pd-isomorphism must take both signs of an edge $\pm i$ to both signs of another edge $\pm j$, but can take $+i$ to $-j$ and vice versa.
\end{definition}
For example $\mathfrak{L}_0 = ((+4,-2,-5,+1),(+2,-6,-3,+5),(+6,-4,-1,+3))$ is a valid pd-code of three vertices and six edges, as each of $\pm 1, \dots, \pm 6$ occurs once, and each pair of non-adjacent labels in a single quadruple (such as $+4,-5$ and $-2,+1$ in the first quadruple) has opposite signs.

This definition is similar to the ``combinatorial maps'' of Walsh and Lehman~\cite{Walsh:1972ti}. Indeed, a pd-code $\mathfrak{L}$ describes an element $\sigma$ of the permutation group on the $4n$ letters $\{\pm1,\pm2,\cdots,\pm2n\}$ which is a product of $n$ disjoint cycles of size $4$, together with an implicit permutation $\tau = (+1 -1)(+2 -2)\cdots(+2n -2n)$ and a ``consistent'' choice of orientations on the edges. This (and other) formulations are explored by Coquereaux, et\ al.~\cite{Coquereaux:2015wv}.
\begin{definition}
Given a pd-code $\mathfrak{L}$ and a signed edge $e$, we can define the \emph{successor edge} $s(e)$ to be minus the edge immediately preceding $e$ in the cyclic ordering of the quadruple where $e$ occurs. The faces of a pd-code $\mathfrak{L}$ are the orbits of the successor map.
\end{definition}
In the pd-code $\mathfrak{L}_0$, for example, $-2$ occurs in the quadruple $(+4,-2,-5,+1)$, so $s(-2) = -(+4) = -4$. Similarly, $s(s(-2)) = s(-4) = -(+6) = -6$, and $s(s(s(-2)) = s(-6) = -(+2) = -2$. Thus $(-2,-4,-6)$ is a face of $\mathfrak{L}_0$.
\begin{proposition}[Mastin, \cite{Mastin:2015ii}]
The vertices, edges, and faces of $\mathfrak{L}$ form a cell-complex structure on a 2-dimensional surface $C(\mathfrak{L})$. Given two pd-codes $\mathfrak{L}$ and $\mathfrak{L'}$, the pd-codes are pd-isomorphic $\iff$ the cell-complexes $C(\mathfrak{L})$ and $C(\mathfrak{L}')$ are isomorphic (the isomorphism may reverse orientation).
\end{proposition}
We can then define
\begin{definition}
The \emph{genus} $g$ of a pd-code $\mathfrak{L}$ is the genus of the associated cell complex $C(\mathfrak{L})$. It is given by $V - E + F = 2 - 2g$, where $V$, $E$, and $F$ count the vertices, edges, and faces of $\mathfrak{L}$. If a pd-code has genus 0, it is \emph{planar}.
\end{definition}
It is another theorem of Mastin that
\begin{proposition}[Mastin~\cite{Mastin:2015ii}]
There is a bijection between $n$-vertex link shadows and $n$-vertex \emph{planar} pd-codes up to pd-isomorphism.
\end{proposition}
The bijection itself is easy to construct: the vertices of the pd-code represent intersections of the circles, the edges represent the sections of circles between intersections (signs represent orientation along the circle), and the cyclic ordering of the vertices is the counterclockwise order of circle arcs at each intersection point. An example of a link shadow and the associated pd-code is shown in Figure~\ref{fig:PD}.
\begin{figure}[H]
\begin{center}
\hphantom{.}
\hfill
\raisebox{-0.5\height}{
\begin{overpic}[height=2cm]{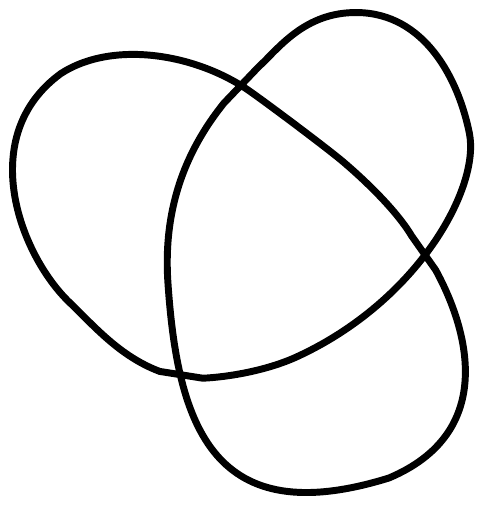}
\put(90,90){$1$}
\put(40,50){$2$}
\put(90,0){$3$}
\put(60,55){$4$}
\put(5,90){$5$}
\put(55,35){$6$}
\end{overpic}}
\hfill
$\{[+1,-5,-2,+4],[+2,+5,-3,-6],[-1,-4,+6,+3]\}$
\hfill
\hphantom{.}
\end{center}
\caption{\label{fig:PD} A three-crossing shadow and its pd-code. There is only one component. Note that we may omit directional arrows as the orientation can be inferred from the ordering of the edge labels.}
\end{figure}
It is clear that there are finitely many $n$-vertex planar pd-codes, and so finitely many equivalence classes of pd-codes up to isomorphism. By the proposition, there are therefore finitely many $n$-vertex link shadows. This will be our initial probability space.

Given a link shadow $\mathfrak{L}$ we can define a link diagram by
\begin{definition}
A \emph{link diagram} is a link shadow where each intersection is decorated with over-under information for the circles meeting at the intersection. We call these intersections \emph{crossings}. The equivalence relation for diagrams is the \emph{diagram isomorphism}: a diffeomorphism (not necessarily orientation-preserving) of $S^2$ to $S^2$ which respects over-under information and preserves the orientation of the curves.
\end{definition}
It is clear that there are at most $2^{\text{\# crossings}}$ link diagrams associated to a given link shadow, but that this number will be smaller if there are nontrivial shadow automorphisms.
\begin{definition}
In the \emph{random diagram model}, a random $n$-crossing knot is selected uniformly from the counting measure on the finite set of one-component $n$-crossing link diagrams.
\end{definition}
This is the smallest probability space including all the knot diagrams that one can define; however, note that we could expand the probability space by doing things like choosing planar, rather than spherical embeddings (hence labeling a face as the exterior) or choosing a basepoint on each component. These amount to making the rules for diagram and shadow isomorphism more strict, increasing the number of equivalence classes of diagrams.

\section{Enumerating shadows}

Our first goal is to enumerate the link shadows. This section describes our two enumeration algorithms. Each is built on the same computational foundation; a library which allows us to manipulate pd-codes. In this library, we provide
\begin{definition}
A \emph{pdstor} is an ordered collection of pd-isomorphism types of pd-codes. We define the operation of adding a pd-code to a pdstor as adding the pd-code if it does not belong to an existing pd-isomorphism type (that is, is not pd-isomorphic to a pd-code present in the pdstor) and ignoring the pd-code otherwise.
\label{def:pdstor}
\end{definition}
Our implementation uses a combination of hashing and brute-force comparison to check whether a given pd-code is already pd-isomorphic to something in the database. Details of the pd-isomorphism check are provided in Appendix A for the curious reader.

We are left with the problem of creating a set of input pd-codes to the pdstor which are guaranteed to cover all $n$-vertex link shadows. We will then separate out the one-component knot shadows by searching the pdstor.

\subsection{Dual quadrangulations and connect sums}
\label{subsec:connect sum}

Brinkmann et al.~\cite{Brinkmann:2007up} provide an algorithm to enumerate all of the simple embedded planar quadrangulations of $S^2$, up to embedded isomorphism. A quadrangulation is a planar graph where each face has four edges, as in Figure \ref{fig:QuadExamples}.
\begin{figure}[H]
	\begin{center}
	\begin{overpic}[width=3in]{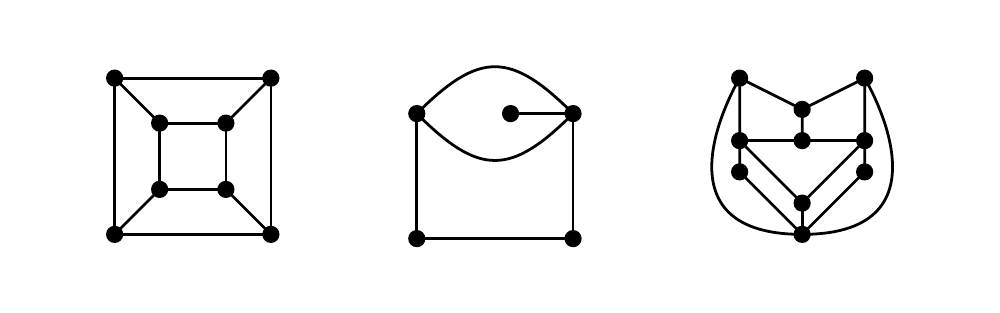}
	\end{overpic}
	\end{center}
	\caption{\label{fig:QuadExamples} This figure shows examples of quadrangulations.}
\end{figure}
The dual graph to a quadrangulation is a connected 4-regular embedded planar multigraph (graphs embedded on a surface sometimes called ``maps''), as shown at left in Figure~\ref{fig:NonSimpleQuad}. In other words, the dual graph defines a link shadow. This is almost enough to enumerate shadows. However, not every link shadow is obtained in this way: if the quadrangulation is simple, no pair of faces in the link shadow share more than one edge. As we can see in Figure \ref{fig:NonSimpleQuad}, this property is not true for every link shadow.
\begin{figure}[H]
\hphantom{.}
\hfill
\includegraphics[height=1in]{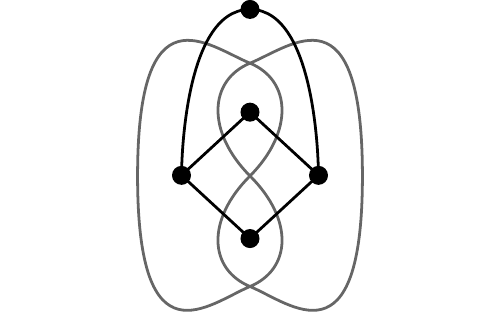} \hfill
\includegraphics[height=1in]{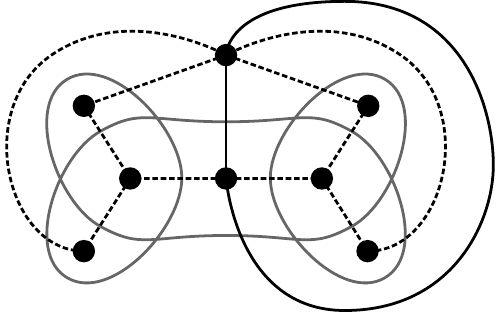}
\hfill
\hphantom{.}
\caption{\label{fig:NonSimpleQuad} This figure shows examples of quadrangulations and their dual graphs, which are link shadows. The quadrangulation at left is simple; the corresponding shadow is prime. The quadrangulation at right is non-simple; the corresponding shadow is composite.}
\end{figure}
We need a familiar idea in a slightly new guise: prime and composite shadows.
\begin{definition}
Every pair of faces in a \emph{prime shadow} shares at most one edge; all other link shadows are \emph{composite shadows}.
\end{definition}
This is not the same thing as prime and composite links-- we can assign crossings and get a prime diagram of a composite link or a composite diagram of a prime link-- but the theory is quite similar. The biggest difference, as we will see below, is that while connect sum for knot types is associative, connect sum for knot shadows is not.

\begin{definition}
Given edges $e$ and $e'$ in pd-codes $\mathfrak{L}$ and $\mathfrak{L}'$, we can construct a new pd-code $\mathfrak{L} \#_{e,e'} \mathfrak{L'}$. The edges of $\mathfrak{L} \#_{e,e'} \mathfrak{L'}$ are the edges of $\mathfrak{L}$ together with the edges of $\mathfrak{L'}$. The vertices of $\mathfrak{L} \#_{e,e'} \mathfrak{L'}$ are the vertices of $\mathfrak{L}$ together with the vertices of $\mathfrak{L'}$ with one change: $+e$ and $+e'$ are swapped ($-e$ and $-e'$ stay in their original positions).
\end{definition}
The effect of the definition is to switch the positions of the heads of the edges $e$ and $e'$ in the crossings where they occur. We are not trying to be too specific about notation because the edges will need to be relabeled in the new pd-code (see our code for one possible implementation).
We can then prove
\begin{proposition}
If $\mathfrak{L}$ and $\mathfrak{L'}$ are planar, then so is the pd-code $\mathfrak{L} \#_{e,e'} \mathfrak{L'}$.
\end{proposition}

\begin{proof} The new collection of vertices clearly still partitions the new collection of signed edge labels, since the same edge labels occur in vertices as in the original pd-codes $\mathfrak{L}$ and $\mathfrak{L'}$. We are swapping a pair of $+$ labels, so the rule on non-adjacent signs is still obeyed in the new pd-code. This means that we need only check the genus.

If we think of faces of a pd-code as lists of signed edges, the effect of the connect sum operation is to concatenate two pairs of these lists. As Figure~\ref{fig:connect sum} shows, the successor of $+e$ in $\mathfrak{L}$ is replaced by the successor of $+e'$ in $\mathfrak{L'}$, and the chain of successors continues as before until returning eventually to the successor of $+e$ in $\mathfrak{L}$. This creates a new face in $\mathfrak{L} \#_{e,e'} \mathfrak{L'}$ which is formed by merging two previous faces in $\mathfrak{L}$ and $\mathfrak{L'}$ and contains $+e$ and $+e'$. We can argue similarly that a second new face of $\mathfrak{L} \#_{e,e'} \mathfrak{L'}$ is also created, containing $-e$ and $-e'$.

\begin{figure}[H]
\hphantom{.}
\centering
\begin{overpic}[height=1in]{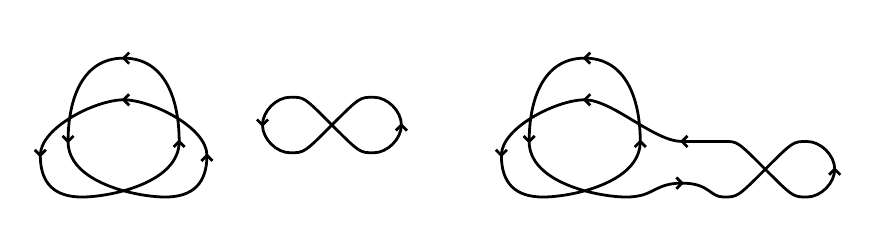}
  \put(20,4){$e_1$}
  \put(26,13){$\#$}
  \put(33,18.5){$e_2$}
  \put(51,13){$=$}
  \put(80,14){$e_1'$}
  \put(83,3){$e_2'$}
\end{overpic}

\caption{\label{fig:connect sum} The connected sum of two link shadows yields a new shadow where two pairs of faces have been merged.}
\end{figure}
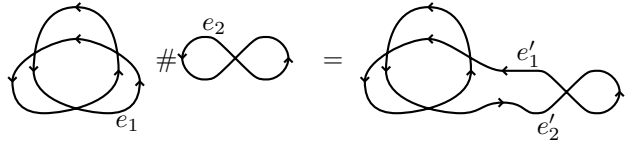
Since the number of vertices and edges of $\mathfrak{L} \#_{e,e'} \mathfrak{L'}$ is the sum of the vertices and edges in $\mathfrak{L}$ and $\mathfrak{L'}$ and the number of faces is 2 less than the sum of the number of faces in $\mathfrak{L}$ and $\mathfrak{L'}$, $V - E + F = 2$ for $\mathfrak{L} \#_{e,e'} \mathfrak{L'}$, so the connect sum pd-code is planar as desired.
\end{proof}

We now give a theorem corresponding to the prime decomposition of links~\cite{MR0102821}:
\begin{proposition}
Every composite link shadow can be created by connect sum operations on a well-defined set of prime link shadows called the prime factors of the composite shadow.
\label{prop:prime decomposition}
\end{proposition}

\begin{proof}
We proceed by induction on the number of pairs of edges shared by two faces. If no two faces share more than one edge, the number of pairs is zero and the diagram is already prime. This is the base case.

Observe that as in Figure~\ref{fig:undo connect sum}, if some pair of faces shares some number of edges, the shared edges occur in precisely opposite order on both faces (if not, the diagram isn't planar). This means that our diagram is the connect sum of two subdiagrams created by cutting and splicing a pair of adjacent shared edges, as on the right side of that Figure.
\begin{figure}[H]
\hphantom{.}
\hfill
\includegraphics[height=1in]{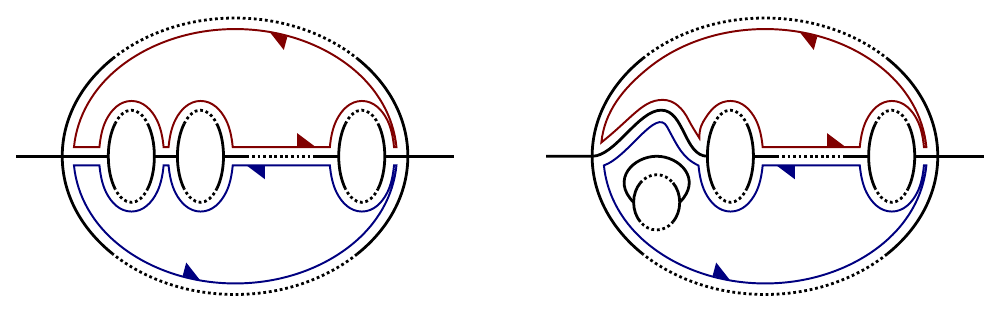}\hfill
\hphantom{.}
\caption{\label{fig:undo connect sum} Two faces which share more than one edge must create an opportunity for a cut and splice.}
\end{figure}

Since each subdiagram has fewer pairs of edges shared by two faces, by induction we may write them as connect sums of collections of prime diagrams. But we made arbitrary choices when deciding which pair of adjacent edges to cut-and-splice, and so must check that the collection of prime diagrams did not depend on these choices.

It's enough to show that the collection of subdiagrams produced by doing any two such cut-and-splices in one order is the same as the collection of subdiagrams produced by doing them in the other. But this is clear as the operations don't interfere with one another (Figure~\ref{fig:splices commute} illustrates the point.)\end{proof}

\begin{figure}[H]
\hphantom{.}
\hfill
\begin{overpic}[height=1in]{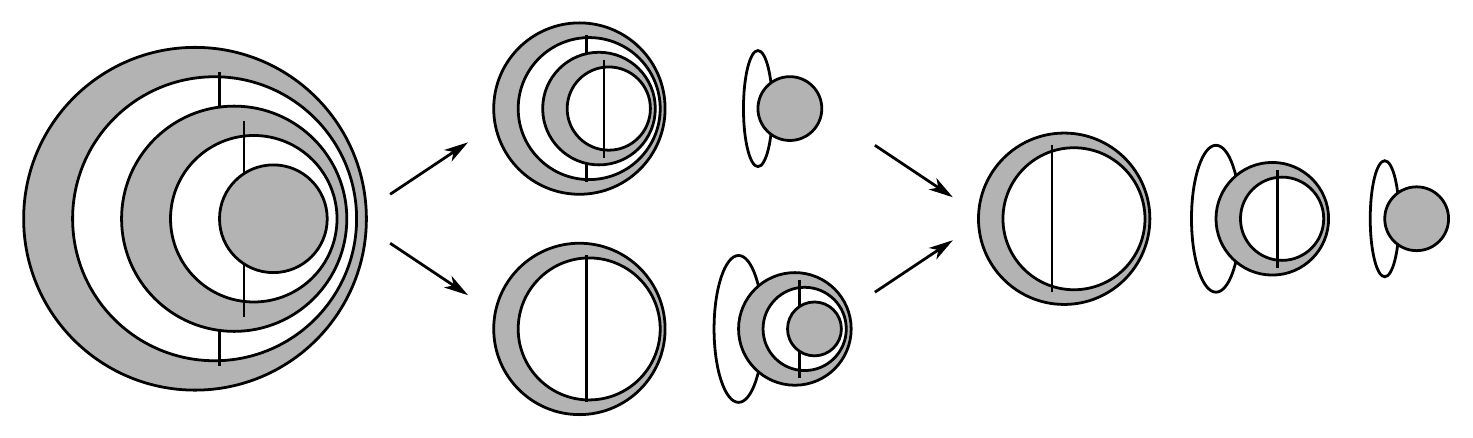}
  \put(2.5,14){$A$}
  \put(9,14){$B$}
  \put(16,14){$C$}
  \put(8,0){$A\#B\#C$}
  \put(37,30){$A\#B$}
  \put(38.5,-1){$A$}
  \put(52,30){$C$}
  \put(50,-1){$B\#C$}
  \put(72,6){$A$}
  \put(83,6){$B$}
  \put(94,6){$C$}
\end{overpic}
\hfill
\hphantom{.}
\caption{\label{fig:splices commute} Cutting and splicing at two locations can be done in either order and produce the same list of subdiagrams.}
\end{figure}
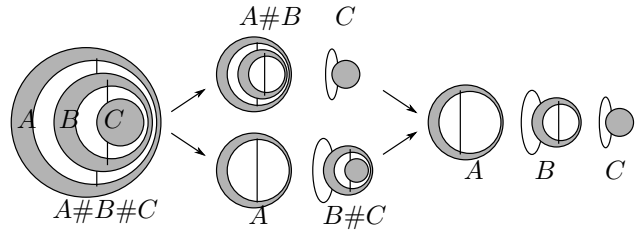

We now know that every $n$ vertex composite shadow $L$ is the connect sum of prime shadows $L_1, \dots, L_k$ and that
the numbers of vertices $V_1, \dots, V_k$ form an integer partition of $n$. Therefore, we need to enumerate ``all connect sums of all link shadows with numbers of vertices that partition $n$''. This is not as trivial as it sounds. Writing the connect sum as $L_1 \# \dots \# L_k$ is dangerously misleading-- though connect sum is associative (and even commutative) on isotopy classes of knots, the same is not true for links and for diagrams. Figure~\ref{fig:not associative} gives an example of a  link shadow $L$ which is a connect sum of diagrams $L_1$, $L_2$ and $L_3$, which \emph{can} be written $(L_1 \# L_2) \# L_3$ but \emph{can't} be written $L_1 \# (L_2 \# L_3)$.
\begin{figure}[H]
\hphantom{.}
\hfill
\begin{overpic}{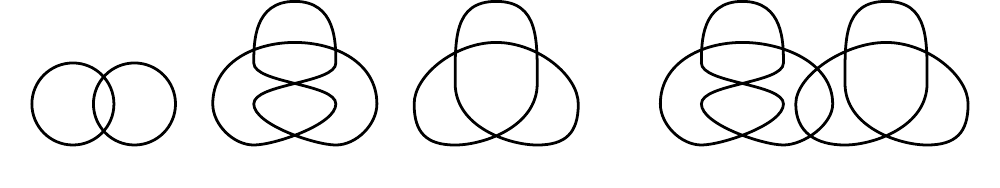}
\put(8.5,0){$L_1$}
\put(28,0){$L_2$}
\put(48,0){$L_3$}
\put(80,0){$L$}
\end{overpic}
\hfill
\hphantom{.}
\caption{\label{fig:not associative} $L$ can be written as $L = (L_1 \# L_2) \# L_3$ but not as $L_1 \# (L_2 \# L_3)$.}
\end{figure}
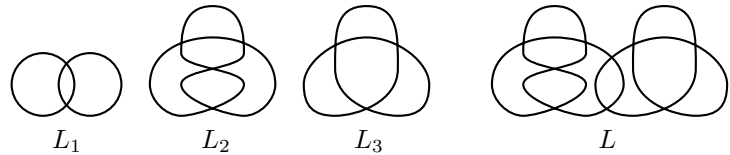
It is easy to see the following.

\begin{lemma}
The set of prime summands $\{L_1, \dots, L_k\}$ of a composite link shadow $L$ can be renumbered so that $L = (((L_1 \# L_2) \# L_3) \# \cdots L_k)$ and $L_1$ has the largest number of vertices among all the $L_1, \dots, L_k$. The numbers of vertices $V_2, \dots, V_k$ in the other prime summands are \emph{not} guaranteed to be in sorted order.
\end{lemma}

\begin{proof}
Since the overall shadow is connected, we can start connect-summing at any prime summand and build the rest of the shadow from there by undoing the cut-and-splice operations of Proposition~\ref{prop:prime decomposition}. Therefore, we are free to choose $L_1$ to have a maximal number of vertices.
\end{proof}

\begin{lemma}
The set $\{L_1, \dots, L_k\}$ of the prime summands of $L$ is a shadow-isomorphism invariant of $L$. In particular, the set of numbers of vertices $\{V_1, \dots, V_k\}$ is a shadow isomorphism invariant of $L$.
\label{lem:prime summands invariant}
\end{lemma}

\begin{proof} This follows directly from that the set $\{L_1, \dots, L_k\}$ is well-defined. \end{proof}

Now recall that a pdstor (Definition~\ref{def:pdstor}) is an ordered collection of canonical pd-codes which represent pd-isomorphism classes. We can define connect sum for pdstors by saying that $P_1 \# P_2$ is the pdstor constructed by adding all $\mathfrak{L_1} \#_{e1,e2} \mathfrak{L_2}$ where $e_i$ is an edge of some $\mathfrak{L_i} \in P_i$. We now generate composite link shadows iteratively:

\begin{algorithmic}
\Procedure{BuildCompositeShadows}{$n$}
\Comment{Build all link shadows with $\leq n$ vertices}

\State Use \emph{plantri} to build pdstors $P_1, \dots, P_n$ containing all oriented prime link shadows with $1$ to $n$ vertices.

\ForAll{$1 \leq k \leq n$}

 \State Define an empty pdstor $P_k$.
 \ForAll{partially sorted partitions $k_1k_2 \dots k_\ell$ of $k$ with $k_1 \geq k_i$}
    \State  Build the pdstor $P_{k_1k_2\cdots k_\ell} = ((P_{k_1} \# P_{k_2}) \# P_{k_3}) \# \cdots P_{k_\ell}$.
    \State  (If $((P_{k_1} \# P_{k_2}) \# P_{k_3}) \# \cdots P_{k_{\ell-1}}$ was already computed, we can reuse it.)
 \EndFor

 \ForAll{\emph{fully} sorted partitions $k_1k_2 \dots k_\ell$ of $k$ with $k_1 \geq k_2 \geq \dots \geq k_\ell$}
    \State Define an empty pdstor $P$.
    \ForAll{\emph{partially} sorted partitions $k'_1k'_2\cdots k'_\ell$ with the same \emph{set} of $k_i$}
       \State Add all elements of $P_{k'_1k'_2\cdots k'_\ell}$ to $P$.
      \State (Some of these will be isomorphic to one another.)
    \EndFor
    \State Add all elements of $P$ to $P_k$.
    \State (By Lemma~\ref{lem:prime summands invariant}, these aren't isomorphic to anything previously computed,
    \State so we don't need to check for isomorphism with existing elements of $P_k$.)
 \EndFor

\EndFor
\EndProcedure
\end{algorithmic}

\subsection{Expansions of embedded planar simple graphs}

Our next strategy will be much more complicated (and somewhat slower to run), but it serves as crucial check on the previous computation. The basic idea is to define a smaller class of graphs so that the graphs we are interested in can be obtained from the base class of graphs by various expansion moves. Lehel~\cite{JGT:JGT3190050412} gave a strategy for generating all 4-regular graphs in this way from the octahedral graph. Instead of using Lehel's strategy directly, we build on the method of Brinkmann and McKay~\cite{Brinkmann:2007up,McKay:1998wa} for enumerating isomorph-free embedded planar graphs; we extend their work here to generate the class of graphs that we're interested in.

We observed above that the link shadows are embedded isomorphism classes of $4$-regular embedded planar multigraphs. We now define four expansion moves of embedded planar graphs with vertex degree $\leq 4$ which generate embedded planar multigraphs of vertex degree $\leq 4$ with the same number of vertices, but additional edges:
\begin{definition}
The four expansion operations that we will use are the following:
\begin{itemize}
\item $\loopinsert$ loop insertion adds a loop edge to a vertex of degree 1 or 2, as below.  Loop insertion can be performed on each side of a vertex of degree 2.
  \begin{figure}[H]
      \begin{center}
      \includegraphics[height=0.75in]{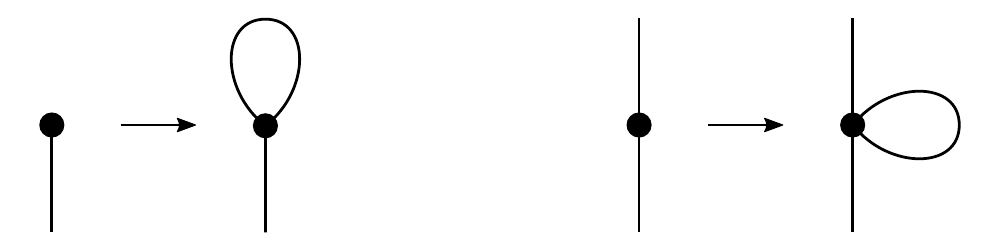}
      \end{center}
  \caption{Adding a loop edge to a vertex of degree 1 or 2.}
  \label{fig:e1 loop insertion}
  \end{figure}
\item $\edgedouble$ reversing edge doubling duplicates an existing edge joining vertices of degree $< 4$ so as to create a new bigon face. Note that the (counterclockwise) order of the two vertices is reversed on the two vertices.
  \begin{figure}[H]
    \begin{center}
    	   \includegraphics[height=0.75in]{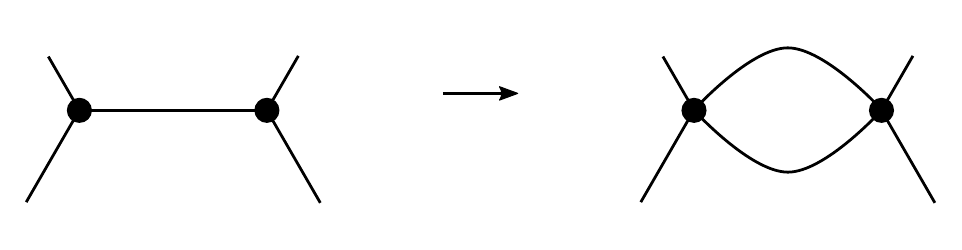}
	\end{center}
	\caption{Doubling an edge joining two vertices of degree $<4$.}
	\label{fig:e2 reversing edge doubling}
   \end{figure}
\item $\cutedgedouble$ preserving doubling also duplicates an existing edge joining vertices of degree $<4$, but
keeps the counterclockwise order of the edges the same on each at each of the two vertices. This sort of doubling is only available if the original edge is a cut edge of the graph.
\begin{figure}[H]
  \begin{center}
    \includegraphics[height=0.75in]{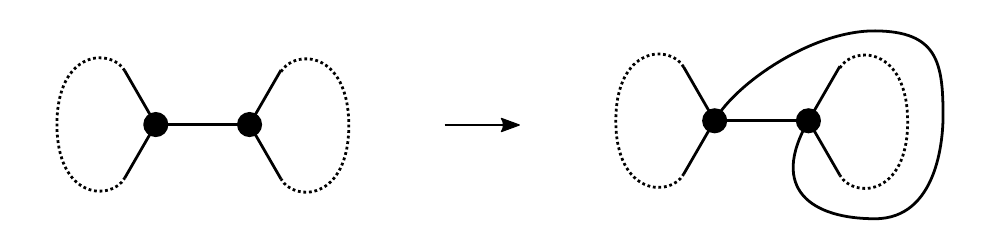}
  \end{center}
  \caption{Doubling a cut edge joining two vertices of degree $<4$ can be done another way.}
\end{figure}
\item $\pairinsert$ pair insertion adds a pair of edges simultaneously, joining two vertices of degree 2 which are both on two faces of the embedding, as below.
  \begin{figure}[H]
  \begin{center}
    \includegraphics[height=0.75in]{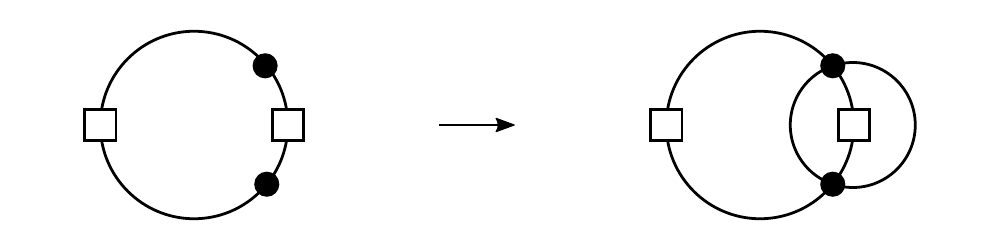}
  \end{center}
  \caption{Adding two new edges to join vertices of degree 2 on the same face.}
  \end{figure}
\end{itemize}
\end{definition}

We can now show

\begin{proposition}
Every link shadow $L$ can be obtained from a connected, embedded planar simple graph of vertex degree $\leq 4$ $G_0$ by a series of $\loopinsert$, $\edgedouble$, $\cutedgedouble$, and $\pairinsert$ expansions.

Equivalently, any link shadow $L$ can be reduced to a connected embedded planar simple graph $L_0$ of vertex degree $\leq 4$ by a series of $\loopinsert$, $\edgedouble$, $\cutedgedouble$, and $\pairinsert$ reductions. The embedded isomorphism type of $L_0$ is determined uniquely by the (unoriented) shadow isomorphism type of $L$ (the order in which the reductions are performed doesn't matter).
\label{prop:reduce}
\end{proposition}

An illustration of the process we describe is shown in Figure~\ref{fig:collapse multigraph}.
\begin{figure}[H]
\begin{center}
\includegraphics[width=4in]{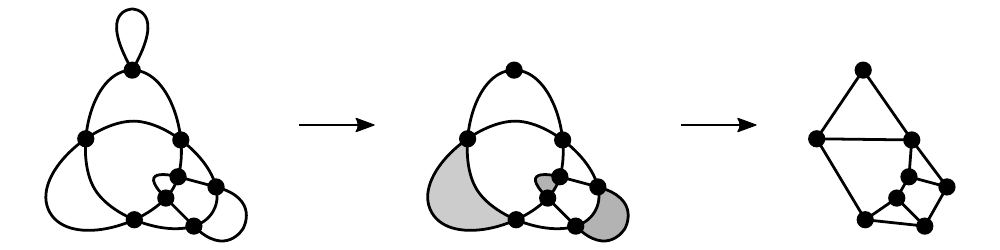}
\end{center}
\caption{\label{fig:collapse multigraph} Any link shadow can be reduced to a connected, embedded planar simple graph of vertex degree $\leq 4$ by a series of reductions, according to Proposition \ref{prop:reduce}.}
\end{figure}

The proof appears in Appendix~\ref{app:reduction proof}. We can now build the link shadows by applying expansions to the graphs produced by~\emph{plantri}. This is not trivial. First, not all sequences of expansion moves lead to link shadows. Second,
different sequences of expansion moves may produce isomorphic link shadows. Some examples are shown in
Figure~\ref{fig:expansion woes}.
\begin{figure}[H]
\begin{center}
\includegraphics[width=4in]{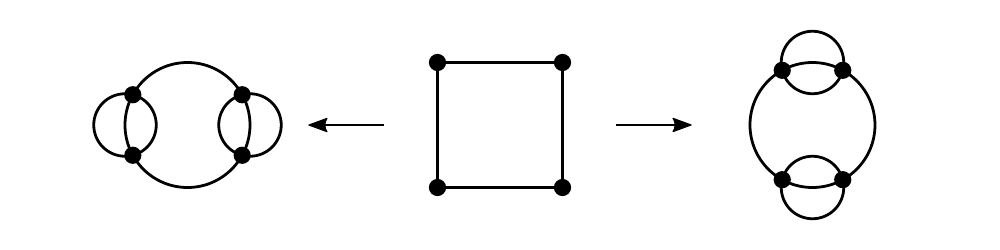}
\end{center}
\caption{Two different sequences of $\edgedouble$ moves lead to isomorphic expansions of this simple graph. This means that the output of the expansion process will contain duplicate link shadows which must be detected and eliminated.}
\label{fig:expansion woes}
\end{figure}

We can think of this as the problem of searching for collections of expansion moves which obey various equations, such as the fact that the total vertex degree must be 4 in any complete solution (there are less obvious equations as well-- the complete system is specified in Appendix~\ref{app:branch and bound}).

We find all of the solutions using a standard branch-and-bound algorithm and add the results to a pdstor to eliminate isomorphic pd-codes. We can save a lot of time in this process by noting that~Proposition~\ref{prop:reduce} tells us that we need only check expansions of the same graph (and their reorientations) against each other for potential pd-isomorphisms.

\section{Computing and verifying the results}

We implemented the above algorithms in C and used them to generate a list of link shadows up to 10 crossings, each unique up to pd-isomorphism. The longest run took several days on a desktop computer. We then extracted the knot shadows, leaving the link shadows for future work. We assigned orientations and over-under information at the crossings to each shadow to make a list of diagrams containing various duplicate diagrams due to symmetry. We eliminated diagram-isomorphic duplicates from this list to arrive at a database of diagrams. Many knots were identified uniquely by their HOMFLY polynomial, which we computed using the Millett/Ewing HOMFLY-PT code~\cite{MR98k:57010}. About 6.5 million diagrams whose HOMFLY corresponded to more than one knot were classified by using Bar-Natan and Morrison's KnotTheory \emph{Mathematica} package to compute the Kauffman polynomial and signature. Since KnotTheory gives incorrect answers for diagrams with one-component ``loop'' faces, we had to eliminate these faces before computation.

This part of the computation was large enough to require organization; several thousand hours of computer time were required to expand the 10 crossing shadows into diagrams and compute their HOMFLY polynomials. We first divided the 10 crossing shadows into roughly 10,000 pieces, each corresponding to about 20 minutes of computation time. These input files were stored on the Amazon S3 storage service. We then entered a message for each input file into a queue in the Amazon SQS service. Worker processors read job descriptions from the queue; SQS then placed a temporary hold on the messages. If the job completed, the workers deleted the job message from the queue and uploaded their results to S3; if jobs failed to complete, SQS placed the message back into the queue after a delay of one hour. The worker processes ran on a mixture of local hardware and virtual Linux machines running in Amazon computing centers in Oregon and Virginia. Time on the virtual machines was obtained by bidding for an hourly price for computation. We paid an average of $0.4$ cents/CPU hour for virtual machines and were able to lease 400 simultaneous cores at this price. The counts of shadows and diagrams are given in Table~\ref{tab:counts}.

\begin{table}[H]
\begin{ruledtabular}
\begin{tabular}{llllll}
$\Cr$ & Prime Shadows & Link Shadows    & Knot Shadows     & Knot Diagrams \\ \hline
3  & $\num{1}$   & $\num{7}$         & $\num{6}$        & $\num{36}^*$ \\
4  & $\num{2}$   & $\num{30}$        & $\num{19}$       & $\num{276}^*$ \\
5  & $\num{3}$   & $\num{124}$       & $\num{76}$       & $\num{2936}^*$ \\
6  & $\num{9}$   & $\num{733}$       & $\num{376}$      & $\num{35872}^*$ \\
7  & $\num{18}$  & $\num{4586}$      & $\num{2194}$     & $\num{484088}^*$ \\
8  & $\num{62}$  & $\num{33373}$     & $\num{14614}$    & $\num{6967942}^*$ \\
9  & $\num{198}$ & $\num{259434}^*$  & $\num{106421}$   & $\num{105555336}^*$ \\
10 & $\num{803}$ & $\num{2152298}^*$ & $\num{823832}$   & $\num{1664142836}^*$ \\
\end{tabular}
\end{ruledtabular}
\caption{The number of link (including knots) and knot shadows and diagrams through 10 crossings. The unstarred numbers in the column of prime shadows are sequence A113201 in the OEIS. The unstarred numbers in the table of link shadows matche K\'apolnai et al.~\cite{Kapolnai:2012hs}. The numbers in the column of knot shadows are sequence A008989 in the OEIS. They match the ``UU, $g=0$'' row on page 43 of the the recent preprint of Coquereaux et al.~\cite{Coquereaux:2015wv}, including the value $\num{823832}$ for $\Cr=10$ which those authors give as ``should be confirmed''. The starred numbers are new.}
\label{tab:counts}
\end{table}

The most important question, of course, is how the computation was checked. Our implementation was careful and involved quite a bit of internal self-checking as well as testing against \valgrind for memory problems, and writing a suite of unit tests for the codebase. However, good programming practices can only provide a limited measure of confidence in the results, so we continued to test our work. The first and most important test was to verify that the lists of knot and link shadows obtained by connect summing and by expansions were identical.

We were also able to check against some existing enumerations. K\'apolnai et al.\ classified spherical ``multiquadrangulations''~\cite{Kapolnai:2012hs}, which are the duals of our diagrams as explained in Section~\ref{subsec:connect sum}. Their table 2 of the count of multiquadrangulations matches our count of diagrams exactly through $8$ crossings (note that the quadrangulation has two more vertices than the dual knot diagram has crossings, so their data is shifted by two). It's worth observing that these authors also use \emph{plantri}, so their results are not completely independent from ours. Still, it provides some comfort to see that their implementation on top of \plantri produces the same results as ours. For knot shadows in particular, Arnol'd has given counts of the number of immersions of the unoriented circle into the unoriented sphere with $n$ crossings for $n$ from $0$ to $5$~\cite[page 79]{Arnold:1994wr}. This is sequence A008989 in the Online Encyclopedia of Integer Sequences, with an extension to $n=7$ credited to Guy H.\ Valette. We could not find a published reference for Vallette's extension of the table, but our data for knots does match A008989 including the extension.  Coquereaux et al.~\cite{Coquereaux:2015wv} have recently extended Valette's count using other means (their count is independent of \emph{plantri}) and our numbers also match those in that paper.

Looking at Table~\ref{tab:counts}, one is struck by how close the number of distinct knot diagrams is to the maximum number $2^{\Cr+1} \times (\# \text{ knot shadows})$. To take 8 crossing diagrams as an example, we would expect at most $\num{7482368} =\num{14614} \times 2^9$, and we have $\num{6967942}$-- about $93\%$ of the maximum possible number. By the time we reach 10 crossing diagrams, the corresponding fraction is roughly $98.6\%$. We have fewer distinct diagrams only because some of the underlying knot diagrams have symmetries. For instance, the trefoil diagram $\tref$ has a 3-fold rotational symmetry, so the crossing sign assignments $+--$, $-+-$ and $--+$ are all the same. However, our computations reveal that such symmetries quickly become very rare as the number of crossings increases. Table~\ref{tab:automorphisms} shows the mean number of automorphisms of a knot shadow, which rapidly decreases to 1 (the identity map), as is known for sufficiently large numbers of crossings~\cite{Chapman2015knotasymp}.

\begin{table}[H]
\begin{ruledtabular}
\renewcommand{\arraystretch}{1.2}
\begin{tabular}{llllllllll}
Cr & $3$ & $4$ & $5$ & $6$ & $7$ & $8$ & $9$ & $10$ \\ \hline
Mean Automorphisms   &  5  & $\frac{64}{19}$ & $\frac{44}{19}$ & $\frac{159}{94}$ & $\frac{1447}{1097}$ & $\frac{8426}{7307}$ & $\frac{113460}{106421}$ & $\frac{846979}{823832}$ \\
\hphantom{Mean} --- (decimal) & 5.00 & 3.37 & 2.32 & 1.69 & 1.32 & 1.15 & 1.07 & 1.03 \\
\end{tabular}
\end{ruledtabular}
\caption{The mean number of automorphisms decreases rapidly as the number of crossings in the diagram increases. This means that the number of knot diagrams (with crossing signs and orientations) rapidly approaches the maximum allowed by the number of knot shadows.}
\label{tab:automorphisms}
\end{table}

\subsection{Knot Types and Inferred Counts}

Counting knot types required us to be careful about knot symmetries. To review, mirroring crossings and reversing orientation yield a $\Z_2 \times \Z_2$ action on knot types. If a knot is isotopic to its image under a subgroup of this group, it is said to have a symmetry; there are five symmetry types corresponding to the five subgroups of this group: ``none'', ``mirror'', ``reversible'', ``amphichiral'' (the diagonal subgroup), and ``full''. We refer to a collection of knot types related by the group as a ``base knot type''. For instance, the base knot type $3_1$ consists of the two knot types $3_1$ and $3_1^m$ (the mirror image of $3_1$).

We were able to use the HOMFLY-PT and Kauffmann polynomials and the knot signature to classify almost all of the knots whose base type had symmetry ``full'' or ``reversible'', since these invariants distinguish all knots with 10 crossings or less from their mirror images except the $10_{71}$ knot. There are 36 base types with one of the other three symmetries. These were more difficult to classify as classical invariants don't distinguish knots from their reversals. We relied on symmetry to infer the distribution of counts:

\begin{lemma}
If a base knot type $K$ has symmetry type ``amphichiral'' ($K = K^{mr}$, $K^{m} = K^r$, but $K \neq K^r$) or ``mirror'' ($K = K^m$, $K^r = K^{mr}$, but $K \neq K^r$), then the number of diagrams of the two knot types are equal. If $K$ has symmetry type ``none" ($K \neq K^m \neq K^r \neq K^{mr}$) the number of diagrams of each of these four knot types are equal.
\end{lemma}

\begin{proof}
In each case, there is a group action on diagrams (reverse the orientation of a diagram or take the 4-element group generated by reversing crossing signs and orientation) which converts diagrams of one knot of these knot types to a diagrams of the other types. This action extends to diagram-isomorphism classes of diagrams because the action commutes with any given diagram-isomorphism.

It is not always the case that this action is free on diagram-isomorphism classes of diagrams of any knot type. For instance, taking the mirror image of this diagram of the $7_4$ knot \raisebox{-0.4 \height}{\includegraphics[height=0.25in]{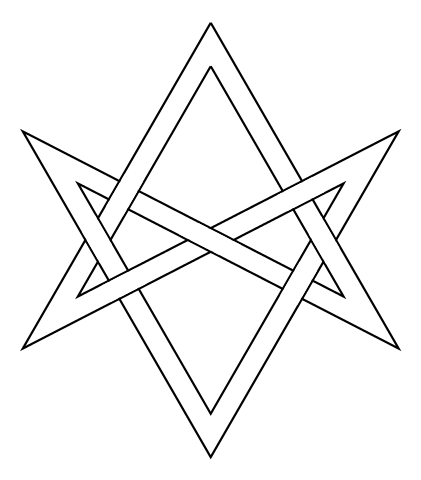}} is equivalent to rotating it by 180 degrees (a diagram-isomorphism). This implies that $7_4$ is a reversible knot, as diagram-isomorphisms are knot isotopies. But by hypothesis, our $K$ is \emph{not} isotopic to its images under the group action, and therefore it cannot by diagram-isomorphic either.

Since the action of the group on this set of diagram-isomorphism classes is free and exchanges the various knot types, the number of diagram-isomorphism classes of diagrams in each knot type must be the same.
\end{proof}

For example, using this lemma we split our original count of $\num{5672}$ 10-crossing diagrams of base type $8_{17}$ into $\num{2836}$ diagrams of type $8_{17}$ and $\num{2836}$ diagrams of type $8_{17}^r$, even though we had no way of knowing which of our diagrams was assigned to each knot type.

The relative frequencies of all the knot types appear in Figure~\ref{fig:knot frequency loglog} in a log-log plot. The figure also shows the 10 most frequent knot types, from the unknot (the most frequent knot type in all our data) and the trefoils $3_1$ and $3_1^m$ through the square knot $3_1 \# 3_1^m$ and the $6_3$ knot, which is the 10th most common knot in all our data. Though we know~\cite{Chapman2015knotasymp} both that unknotted diagrams eventually become exponentially rare, and that the rank-order of the knot types must eventually change, we do not see this effect in our data; the list of 10 most common knots is the same for crossing numbers 6-10. Table~\ref{tab:knot frequency raw data} gives more detailed frequency data for the 40 most common knot types (among 10 crossing diagrams). This data does show some reordering of knot types as crossing number increases.

\begin{table}[H]
\begin{ruledtabular}
\begin{tabular}{ccccccccc}
Crossing Number &
3 & 4 & 5 & 6 & 7 & 8 & 9 & 10 \\
Number of Diagrams &
$\num{36}$ &
$\num{276}$ &
$\num{2936}$ &
$\num{35872}$ &
$\num{484088}$ &
$\num{6967942}$ &
$\num{105555336}$ &
$\num{1664142836}$\\ \hline
$0_1$ &
$\num{34}$ &
$\num{265}$ &
$\num{2744}$ &
$\num{32456}$ &
$\num{422332}$ &
$\num{5852832}$ &
$\num{85253534}$ &
$\num{1291291155}$\\
$3_{1}^{}$ &
$\num{1}$ &
$\num{5}$ &
$\num{85}$ &
$\num{1466}$ &
$\num{25432}$ &
$\num{440570}$ &
$\num{7696083}$ &
$\num{135702456}$\\
$3_{1}^{m}$ &
$\num{1}$ &
$\num{5}$ &
$\num{85}$ &
$\num{1466}$ &
$\num{25432}$ &
$\num{440570}$ &
$\num{7696083}$ &
$\num{135702456}$\\
$4_{1}^{}$ &
-- &
$\num{1}$ &
$\num{18}$ &
$\num{412}$ &
$\num{8450}$ &
$\num{165791}$ &
$\num{3175612}$ &
$\num{60146706}$\\
$5_{2}^{}$ &
-- &
-- &
$\num{1}$ &
$\num{24}$ &
$\num{730}$ &
$\num{18075}$ &
$\num{415290}$ &
$\num{9025926}$\\
$5_{2}^{m}$ &
-- &
-- &
$\num{1}$ &
$\num{24}$ &
$\num{730}$ &
$\num{18075}$ &
$\num{415290}$ &
$\num{9025926}$\\
$3_{1}^{}\#3_{1}^{m}$ &
-- &
-- &
-- &
$\num{2}$ &
$\num{112}$ &
$\num{3953}$ &
$\num{113684}$ &
$\num{2923783}$\\
$6_{3}^{}$ &
-- &
-- &
-- &
$\num{2}$ &
$\num{106}$ &
$\num{3515}$ &
$\num{96666}$ &
$\num{2389180}$\\
$6_{2}^{}$ &
-- &
-- &
-- &
$\num{1}$ &
$\num{58}$ &
$\num{2027}$ &
$\num{58354}$ &
$\num{1493624}$\\
$6_{2}^{m}$ &
-- &
-- &
-- &
$\num{1}$ &
$\num{58}$ &
$\num{2027}$ &
$\num{58354}$ &
$\num{1493624}$\\
$3_{1}^{}\#3_{1}^{}$ &
-- &
-- &
-- &
$\num{2}$ &
$\num{58}$ &
$\num{2006}$ &
$\num{56893}$ &
$\num{1461498}$\\
$3_{1}^{m}\#3_{1}^{m}$ &
-- &
-- &
-- &
$\num{2}$ &
$\num{58}$ &
$\num{2006}$ &
$\num{56893}$ &
$\num{1461498}$\\
$6_{1}^{m}$ &
-- &
-- &
-- &
$\num{1}$ &
$\num{34}$ &
$\num{1267}$ &
$\num{38199}$ &
$\num{1015996}$\\
$6_{1}^{}$ &
-- &
-- &
-- &
$\num{1}$ &
$\num{34}$ &
$\num{1267}$ &
$\num{38199}$ &
$\num{1015996}$\\
$3_{1}^{}\#4_{1}^{}$ &
-- &
-- &
-- &
-- &
$\num{8}$ &
$\num{516}$ &
$\num{20458}$ &
$\num{648362}$\\
$3_{1}^{m}\#4_{1}^{}$ &
-- &
-- &
-- &
-- &
$\num{8}$ &
$\num{516}$ &
$\num{20458}$ &
$\num{648362}$\\
$7_{6}^{m}$ &
-- &
-- &
-- &
-- &
$\num{3}$ &
$\num{193}$ &
$\num{7608}$ &
$\num{240121}$\\
$7_{6}^{}$ &
-- &
-- &
-- &
-- &
$\num{3}$ &
$\num{193}$ &
$\num{7608}$ &
$\num{240121}$\\
$7_{7}^{}$ &
-- &
-- &
-- &
-- &
$\num{2}$ &
$\num{124}$ &
$\num{4709}$ &
$\num{144455}$\\
$7_{7}^{m}$ &
-- &
-- &
-- &
-- &
$\num{2}$ &
$\num{124}$ &
$\num{4709}$ &
$\num{144455}$\\
$7_{5}^{m}$ &
-- &
-- &
-- &
-- &
$\num{2}$ &
$\num{102}$ &
$\num{4244}$ &
$\num{138467}$\\
$7_{5}^{}$ &
-- &
-- &
-- &
-- &
$\num{2}$ &
$\num{102}$ &
$\num{4244}$ &
$\num{138467}$\\
$7_{2}^{m}$ &
-- &
-- &
-- &
-- &
$\num{1}$ &
$\num{44}$ &
$\num{2103}$ &
$\num{74739}$\\
$7_{2}^{}$ &
-- &
-- &
-- &
-- &
$\num{1}$ &
$\num{44}$ &
$\num{2103}$ &
$\num{74739}$\\
$7_{3}^{}$ &
-- &
-- &
-- &
-- &
$\num{1}$ &
$\num{39}$ &
$\num{1793}$ &
$\num{62059}$\\
$7_{3}^{m}$ &
-- &
-- &
-- &
-- &
$\num{1}$ &
$\num{39}$ &
$\num{1793}$ &
$\num{62059}$\\
$4_{1}^{}\#4_{1}^{}$ &
-- &
-- &
-- &
-- &
-- &
$\num{20}$ &
$\num{1176}$ &
$\num{51526}$\\
$7_{4}^{m}$ &
-- &
-- &
-- &
-- &
$\num{1}$ &
$\num{36}$ &
$\num{1516}$ &
$\num{49731}$\\
$7_{4}^{}$ &
-- &
-- &
-- &
-- &
$\num{1}$ &
$\num{36}$ &
$\num{1516}$ &
$\num{49731}$\\
$8_{20}^{}$ &
-- &
-- &
-- &
-- &
-- &
$\num{14}$ &
$\num{985}$ &
$\num{41843}$\\
$8_{20}^{m}$ &
-- &
-- &
-- &
-- &
-- &
$\num{14}$ &
$\num{985}$ &
$\num{41843}$\\
$3_{1}^{m}\#5_{2}^{m}$ &
-- &
-- &
-- &
-- &
-- &
$\num{10}$ &
$\num{784}$ &
$\num{36548}$\\
$3_{1}^{}\#5_{2}^{}$ &
-- &
-- &
-- &
-- &
-- &
$\num{10}$ &
$\num{784}$ &
$\num{36548}$\\
$3_{1}^{m}\#5_{2}^{}$ &
-- &
-- &
-- &
-- &
-- &
$\num{10}$ &
$\num{784}$ &
$\num{36544}$\\
$3_{1}^{}\#5_{2}^{m}$ &
-- &
-- &
-- &
-- &
-- &
$\num{10}$ &
$\num{784}$ &
$\num{36544}$\\
$8_{21}^{}$ &
-- &
-- &
-- &
-- &
-- &
$\num{9}$ &
$\num{574}$ &
$\num{24611}$\\
$8_{21}^{m}$ &
-- &
-- &
-- &
-- &
-- &
$\num{9}$ &
$\num{574}$ &
$\num{24611}$\\
$8_{14}^{}$ &
-- &
-- &
-- &
-- &
-- &
$\num{6}$ &
$\num{442}$ &
$\num{19412}$\\
$8_{14}^{m}$ &
-- &
-- &
-- &
-- &
-- &
$\num{6}$ &
$\num{442}$ &
$\num{19412}$\\
$7_{1}^{m}$ &
-- &
-- &
-- &
-- &
$\num{1}$ &
$\num{8}$ &
$\num{444}$ &
$\num{17441}$
\end{tabular}
\end{ruledtabular}
\caption{This table shows the number of diagrams of each knot type among all knot diagrams with between 3 and 10 crossings. The knot types shown are the most common 40 knot types among 10 crossing diagrams, and they appear in the order of their frequency among 10 crossing diagrams. This is not the same rank order for all crossing numbers-- one can observe that $4_1 \# 4_1$ is less common than $7_4$ among 8 crossing diagrams ($\num{20}$ diagrams versus $\num{36}$ diagrams) but more common than $7_4$ among 10 crossing diagrams ($\num{51526}$ diagrams versus $\num{49731}$ diagrams).}
\label{tab:knot frequency raw data}
\end{table}

More interestingly, the log-log plot of ranked knot frequencies for the 622 different possible knot types for 10 crossing diagrams is roughly linear over 9 orders of magnitude in frequency (there are about $1.6 \times 10^9$ diagrams, with the most frequent knot type (the unknot) occurring $1.2 \times 10^9$ times and the least frequent knot types (a 98-way tie among various 10-crossing knots) appearing exactly once. It is not clear to us why this phenomenon should occur in the data: this data is certainly compatible with the hypothesis that there is an (asymptotic) power-law relationship between knot rank and knot probability akin to Zipf's law, but it would take much larger experiments to provide strong statistical support for such a conjecture.

\begin{figure}[H]
\hfill
\begin{overpic}[width=3in]{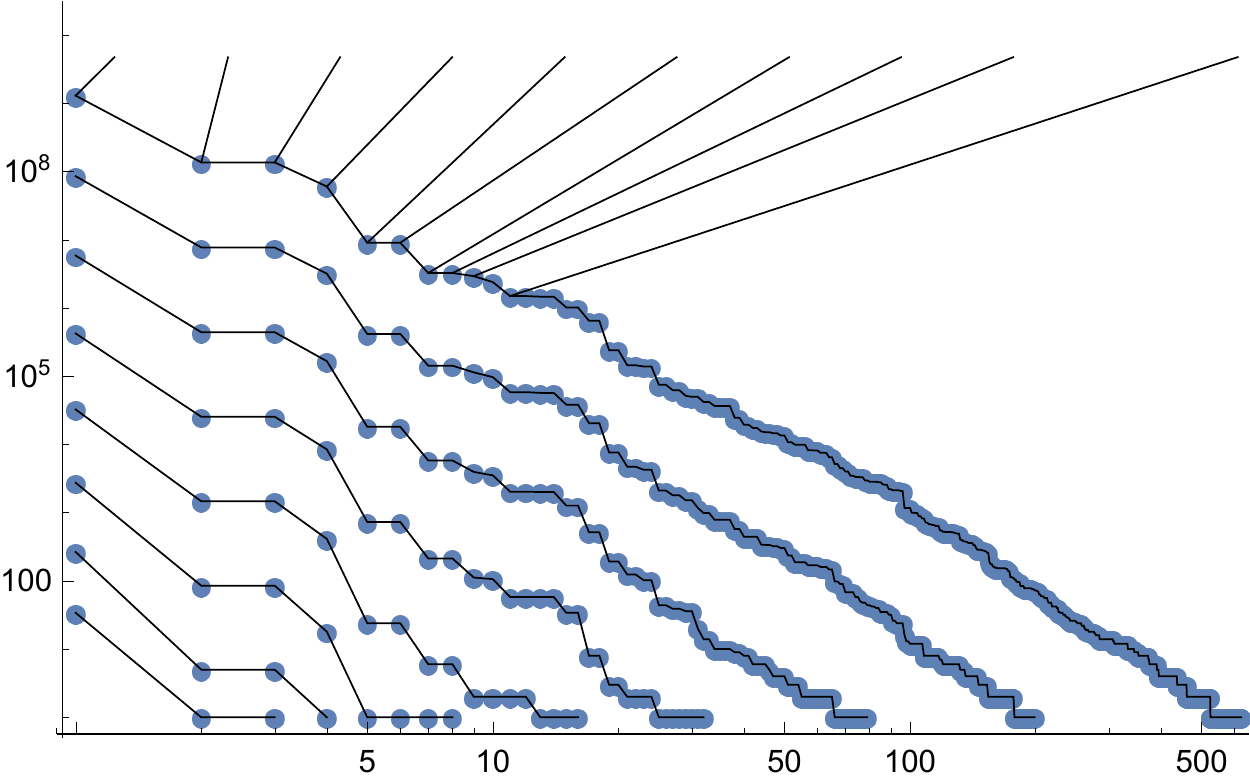}
\put(8,59){$0_1$}
\put(16.5,59){$3_1$}
\put(25,59){$3_1^m$}
\put(34,59){$4_1$}
\put(44,59){$5_2$}
\put(53,59){$5_2^m$}
\put(62,59){$5_1$}
\put(71,59){$5_1^m$}
\put(78,59){$3_1 \# 3_1^m$}
\put(98,59){$6_3$}
\end{overpic}
\hfill
\begin{minipage}[b]{1.5in}
\begin{ruledtabular}
\begin{tabular}{ccc}
$\Cr$ & Unknots & (decimal) \\ \hline
\\ [-0.15in]
 3 & $\frac{\num{17}}{\num{18}}$ & 0.94 \\ [0.05in]
 4 & $\frac{\num{265}}{\num{276}}$ & 0.96 \\[0.05in]
 5 & $\frac{\num{343}}{\num{367}}$ & 0.93 \\[0.05in]
 6 & $\frac{\num{4057}}{\num{4484}}$ & 0.90 \\[0.05in]
 7 & $\frac{\num{105583}}{\num{121022}}$ & 0.87 \\[0.05in]
 8 & $\frac{\num{2926416}}{\num{3483971}}$ & 0.84  \\[0.05in]
 9 & $\frac{\num{42626767}}{\num{52777668}}$ & 0.81 \\[0.05in]
 10 & $\frac{\num{1291291155}}{\num{1664142836}}$ & 0.78
\end{tabular}
\end{ruledtabular}
\end{minipage}
\hfill
\hphantom{.}

\caption{A log-log plot of knot frequencies in rank order for crossing numbers 3 through 10, with the 10 most common knot types identified. These knot types have the same rank ordering in all the crossing numbers computed. The data show some evidence of power-law behavior. At right, we see the unknot fraction in table form.}
\label{fig:knot frequency loglog}
\end{figure}
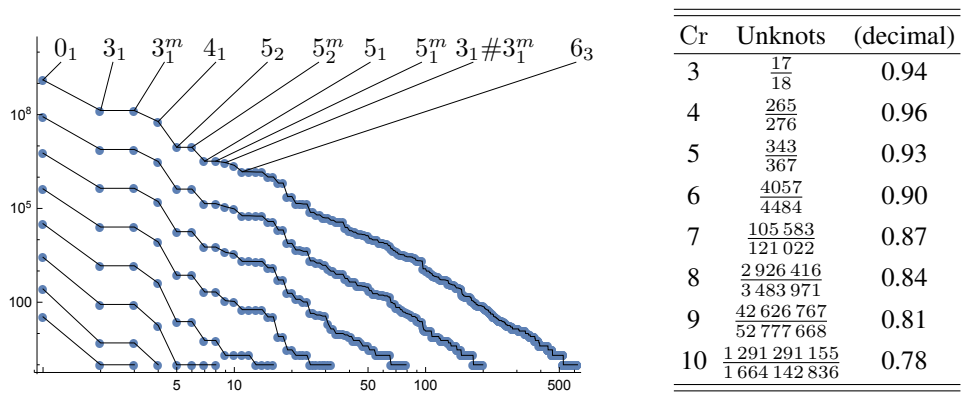

At right in Figure~\ref{fig:knot frequency loglog} we give the unknot fraction explicitly in tabular form. We now try to understand this fraction. Of course, for any shadow we can set the crossings to produce an unknot by going ``downhill'' from some distinguished edge, but this effect can only account for a tiny fraction of the unknots-- there are at most $n$ ways to assign crossings in this manner but there are $~2^n$ different crossing assignments overall. It turns out, however, that structural properties of the set of shadows explain most of the unknot fraction. To see how, we start with Figure~\ref{fig:knot grid}, which shows the first 42 shadows obtained in the enumeration.

It is immediately clear from Figure~\ref{fig:knot grid} that many of the shadows can be simplified by a Reidemeister I move; they have a face with only one edge (a monogon). In fact, we can see from Table~\ref{tab:monogons and bigons} that almost every diagram has at least one monogon and that the mean number of monogons seems to be rising linearly with $n$.

\begin{figure}[H]
\includegraphics[width=6in]{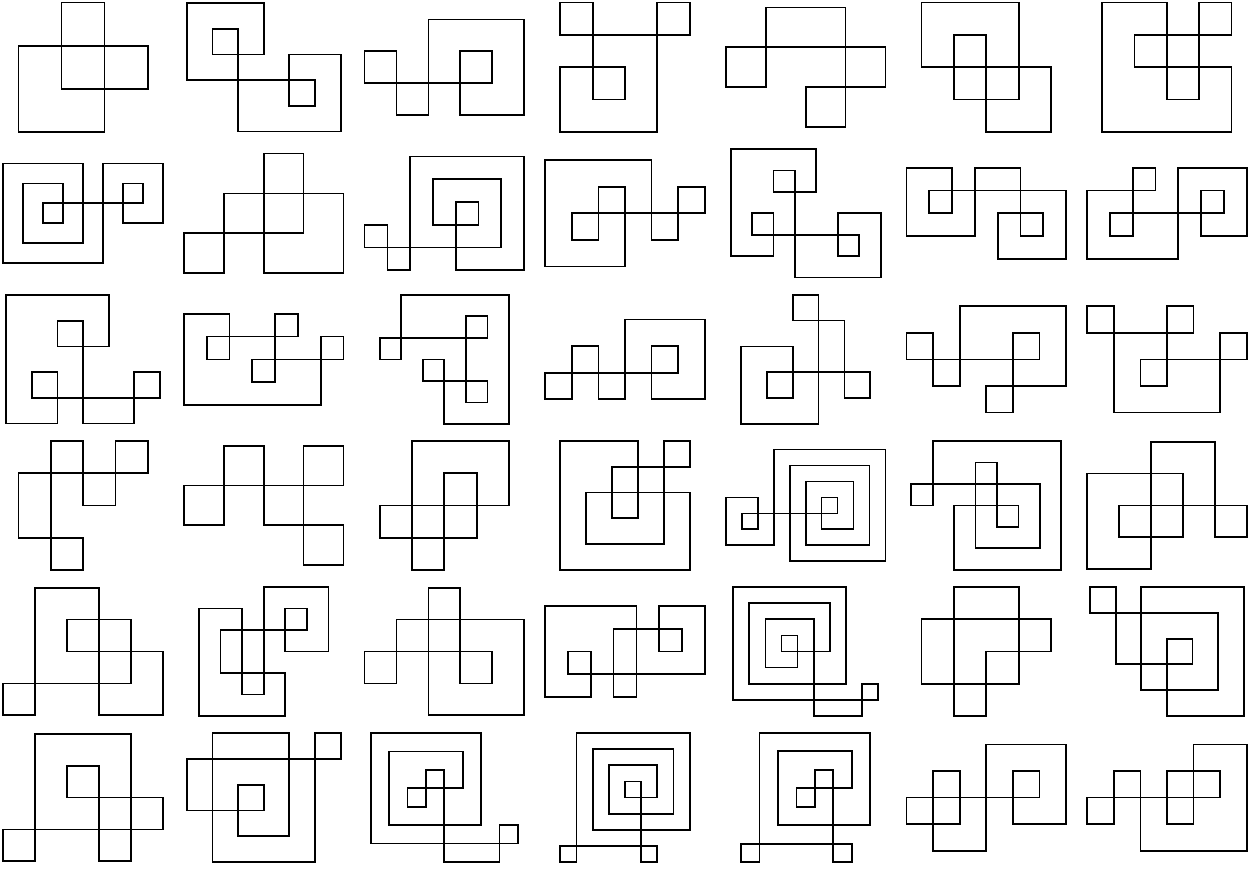}
\caption{The first 42 knot shadows.}
\label{fig:knot grid}
\end{figure}

\begin{table}[H]
\begin{ruledtabular}
\renewcommand{\arraystretch}{1.2}
\begin{tabular}{lllllllll}
Cr                 & $3$ & $4$             & $5$            & $6$                        & $7$ & $8$ & $9$ & $10$ \\ \hline
Mean $\#$ monogons & $\frac{\num{12}}{\num{6}}$ & $\frac{\num{48}}{\num{19}}$ & $\frac{\num{213}}{\num{76}}$ & $\frac{\num{1196}}{\num{376}}$ & $\frac{\num{7714}}{\num{2194}}$ & $\frac{\num{56540}}{\num{14614}}$ & $\frac{\num{448584}}{\num{106421}}$ & $\frac{\num{3758456}}{\num{823832}}$\\
\hphantom{Mean} --- (decimal) & $2.$ & $2.53$ & $2.8$ & $3.18$ & $3.52$ & $3.87$ & $4.22$ & $4.56$ \\
Monogon fraction   & $\frac{\num{5}}{\num{6}}$ & $\frac{\num{18}}{\num{19}}$ &
$\frac{\num{74}}{\num{76}}$ & $\frac{\num{371}}{\num{376}}$ &
$\frac{\num{2178}}{\num{2194}}$ &
$\frac{\num{14562}}{\num{14614}}$ &
$\frac{\num{106216}}{\num{106421}}$ &
$\frac{\num{822989}}{\num{823832}}$ \\
\hphantom{Mean} --- (decimal) & $0.833$ & $0.947$ & $0.974$ & $0.987$ & $0.993$ & $0.996$ & $0.998$ \
& $0.999$ \\ \hline
Mean $\#$ bigons & $\frac{\num{6}}{\num{6}}$ & $\frac{\num{18}}{\num{19}}$ &
$\frac{\num{88}}{\num{76}}$ & $\frac{\num{470}}{\num{376}}$ &
$\frac{\num{3037}}{\num{2194}}$ & $\frac{\num{21925}}{\num{14614}}$ &
$\frac{\num{173342}}{\num{106421}}$ & $\frac{\num{1450209}}{\num{823832}}$ \\
\hphantom{Mean} --- (decimal) & $1.$ & $0.947$ & $1.16$ & $1.25$ & $1.38$ & $1.5$ & $1.63$ & $1.76$ \\
Bigon fraction & $\frac{\num{3}}{\num{6}}$ & $\frac{\num{11}}{\num{19}}$ &
$\frac{\num{52}}{\num{76}}$ & $\frac{\num{275}}{\num{376}}$ &
$\frac{\num{1714}}{\num{2194}}$ & $\frac{\num{11892}}{\num{14614}}$ &
$\frac{\num{89627}}{\num{106421}}$ & $\frac{\num{712961}}{\num{823832}}$ \\
\hphantom{Mean} --- (decimal) & $0.5$ & $0.579$ & $0.684$ & $0.731$ & $0.781$ & $0.814$ & $0.842$ & $0.865$ \\
\end{tabular}
\end{ruledtabular}
\caption{The distribution of monogons among shadows of $n$ crossings. The mean number of monogons in an $n$-crossing diagram fits very well to $1.01927 + 0.356111 n$, while the mean number of bigons fits well to $0.563245 + 0.117749 n$ It would be interesting to know the asymptotic growth rate.}
\label{tab:monogons and bigons}
\end{table}

It is interesting to note that almost every diagram contains either a monogon or a bigon-- there are only three diagrams in our dataset without one or the other! These turn out to be the shadows of torus knots and links that Conway designated as $8^*$, $9^*$, and $10^*$. This means that we can expect to reduce the complexity of an average shadow substantially just by eliminating monogons, which can be done regardless of crossing signs. We note here that there are some adjustments to our enumeration scheme which may produce an interesting subset of shadows, but do not provide any proof: If the $\loopinsert$ loop insertion rule is removed from the graph expansion procedure, all shadows without monogons should be tabulated. Similarly, removing the one-crossing $\twistunknot$ diagram from our connect-summing procedure should tabulate all diagrams without nugatory crossings.

Put another way, almost all knot shadows are composite with the one-crossing diagram $\twistunknot$ as a (diagrammatically) prime factor. Figure~\ref{fig:proportiongraphic} shows the fraction of shadows of a given crossing number with a given number of prime summands. If the number of prime summands is equal to the crossing number $n$, the shadow is a connect sum of one-crossing $8$ diagrams-- these diagrams are called ``tree-like'' by Aicardi~\cite{Aicardi:1994uq}.
Tree-like diagrams are colored in dark on the left-hand side of Figure~\ref{fig:proportiongraphic}. There are no $n$-crossing shadows with $n-1$ prime summands (as there are no 2-crossing prime knot diagrams).

\begin{figure}[H]
\hfill
\raisebox{-0.5 \height}{\begin{overpic}[width=1.5in]{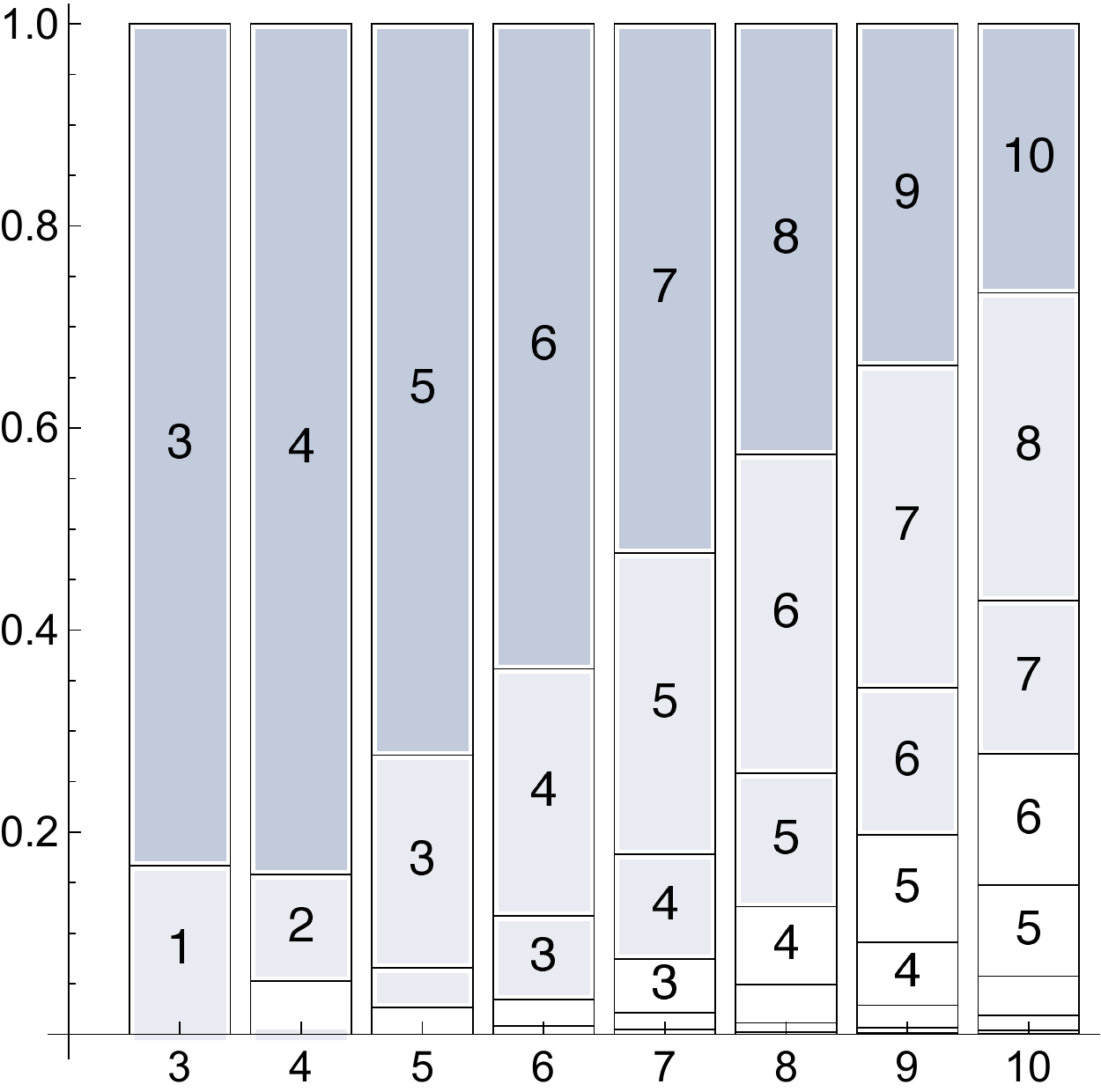}
\end{overpic}
}
\hfill
\raisebox{-0.5 \height}{
\begin{overpic}[width=1.5in]{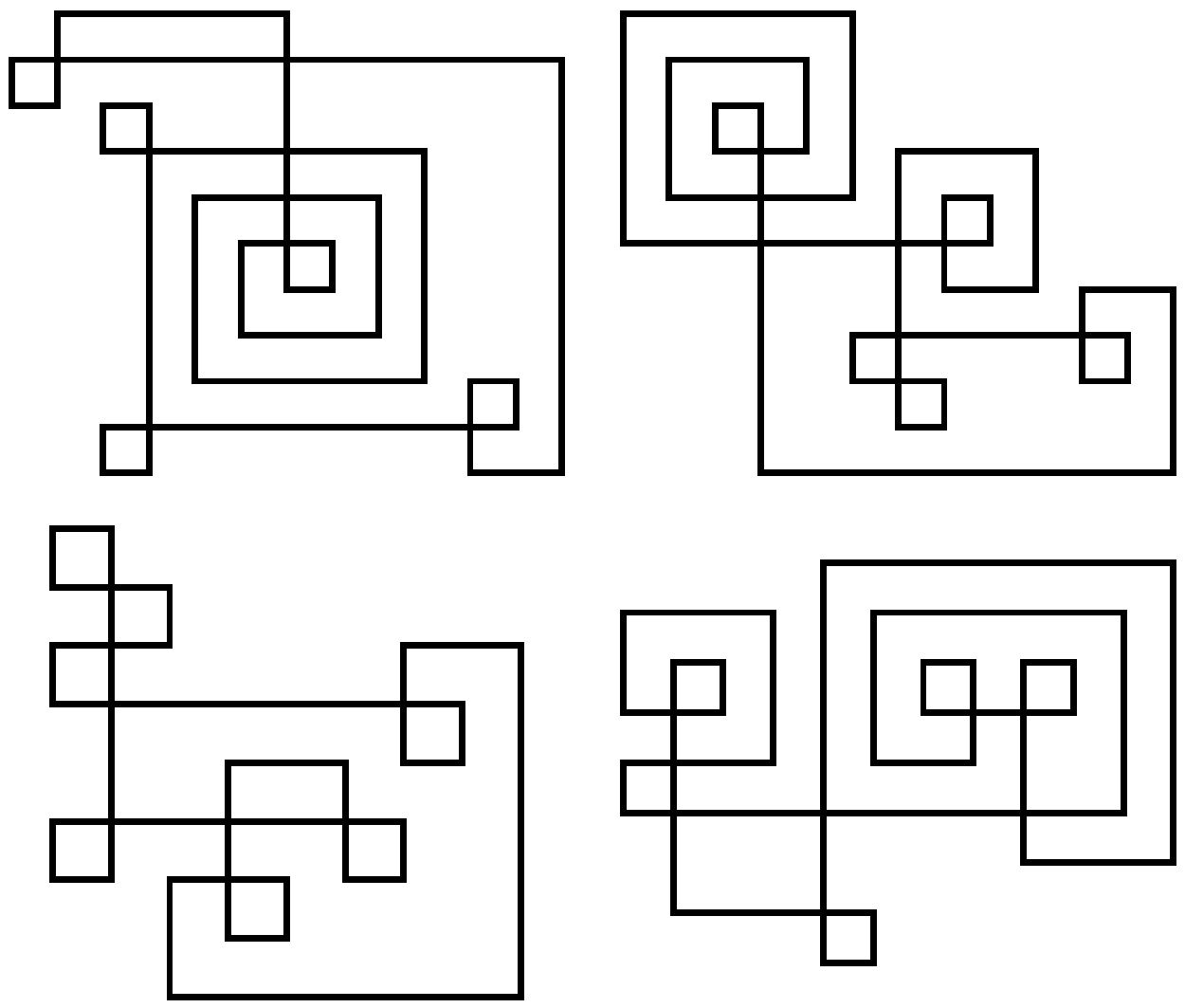}
\end{overpic}
}
\hfill
\hphantom{.}
\caption{Connect sum structure of knot shadows (left, numbers in bars show the number of connect summands), and tree-like 8-crossing shadows (right).}
\label{fig:proportiongraphic}
\end{figure}


We will call the $n$-crossing shadows with $n-2$ prime summands ``3-almost tree-like'' as they are the connect sum of a tree-like diagram and a 3-crossing prime diagram, and similarly call the $n$-crossing diagrams with $n-3$ prime summands ``4-almost tree-like'' as they are the connect sum of a tree-like diagram and a 4-crossing prime diagram. These diagrams are colored in on Figure~\ref{fig:proportiongraphic} in a lighter color.

\begin{proposition}
If a shadow is tree-like, all assignments of crossings result in the unknot; these are the only knot shadows with this property. If a shadow is 3-almost tree-like, $3/4$ of crossing assignments produce unknots. If a shadow is 4-almost tree-like, $7/8$ of crossing assignments produce unknots.
\label{prop:almost treelike mostly unknots}
\end{proposition}

\begin{proof}
The calculations of unknot fractions are easy; a tree-like diagram can be reduced to the unknot by Reidemeister I moves regardless of crossing assignment. Further, there is only one prime knot shadow of 3 or 4 crossings; the assignment of all $+$ or all $-$ crossings yields a $3_1$ or $4_1$ knot, while all other assignments result in unknots.

The statement that the tree-like shadows are the only shadows with the property that all crossing assignments produce unknots is more interesting. Polyak~\cite{Polyak:1998wa} shows that the average value of the Vassiliev $v_2$ invariant (which is $0$ for unknots) over all crossing assignments for a shadow is given by the formula $\frac{1}{8}(\operatorname{J}^+ + 2 \operatorname{St})$ where $\operatorname{J}^+$ and $\operatorname{St}$ are Arnol'd invariants of the shadow (cf. \cite{Arnold:1994wr}). (While $\operatorname{J}^+$ and $\operatorname{St}$ are not spherically invariant, the sum $\operatorname{J}^+ + 2\operatorname{St}$ is.) Further, both he and Aicardi~\cite{Aicardi:1994uq} show that tree-like shadows are the only shadows with $\operatorname{J}^+ + 2 \operatorname{St} = 0$.
\end{proof}

Doing the requisite sums, this analysis predicts that roughly $77\%$ of 8-crossing diagrams, $70\%$ of 9-crossing diagrams, and $63\%$ of 10-crossing diagrams should be tree-like or almost tree-like unknots. We can compare these figures to the actual unknot fractions from Figure~\ref{fig:knot frequency loglog}: roughly $83\%$ of 8-crossing diagrams, $80\%$ of 9-crossing diagrams, and $78\%$ of 10-crossing diagrams are unknots of any kind. This suggests that the treelike phenomenon explains much of the unknot fraction. We note that these figures are not entirely comparable; while the Proposition shows that at least $63\%$ of crossing assignments to 10-crossing shadows result in unknots, some of these crossing assignments might result in diagram-isomorphic diagrams if the underlying shadow has a symmetry, and so the unknot fraction could be slightly smaller. This cannot make much difference to the estimate. The counts in Table~\ref{tab:counts} show that about $98.6\%$ of crossing assignments to 10 crossing shadows produce distinct diagrams, so the worst-case scenario is that this analysis shows that at least $61.4\%$ of 10-crossing diagrams are tree-like or almost tree-like unknots.

\section{Future Directions}

All of our collections of knot shadows are available online~\cite{suppdata}, including files of coordinates for planar embeddings of the shadows as well as their pd-codes. In addition, we include the actual counts of diagrams of various knot types as CSV files, and the files of diagrams with 8 crossings and fewer. We did not assemble files of diagrams for 9 and 10 crossing diagrams, but classified diagrams on-the-fly as we generated them (we estimate the file of 10-crossing diagrams to be about 1.6 tb in size if generated). Our source code is also available.

Once in possession of an enumeration of diagrams, it is tempting to compute knot distances (cf.~\cite{DARCY:2001fi,MoonHyeyoung:2010tm}) and unknotting numbers. We have carried out some preliminary experiments along these lines and are disappointed to report that it does not seem to resolve any of the remaining uncertainties in the knot distance tables.

More sophisticated counting procedures have been applied in recent years, particularly for shadows of knots. Jacobsen and Zinn-Justin\cite{Jacobsen:2002kia} give a transfer matrix analysis for a combinatorial description of prime, reduced, weighted, marked planar shadows for knots very similar to Gusein-Zade's enumeration~\cite{GusenZade:1994wl,GusenZade:1998jz}. Gusein-Zade's actual computer enumeration was carried out only to 10 crossings while Jacobsen and Zinn-Justin push the enumeration to the nearly unthinkable 22-crossing case, where they give an exact count of 40558226664529044000 knot shadows. Of course, their enumeration does not construct each shadow, so it cannot be used easily to estimate the unknot fraction.

Schaeffer~\cite{Schaeffer:1997wo} (see also~\cite{Bouttier:2002iu}) constructed a very insightful bijection between link shadows with a distinguished, directed edge (``rooted link shadows'') and a class of decorated trees called ``blossom trees''. One of us (Chapman), uses this identification to analyze the asymptotic behavior of our model in further work~\cite{Chapman2015knotasymp}, proving that most diagrams are heavily composite and hence represent nontrivial knots.

\acknowledgments
The authors are incredibly grateful to Eric Rawdon, who graciously checked our computations of knot types for the various diagrams using his own knot classification code. We also would like to acknowledge the hospitality of the Issac Newton Institute, where much of the original code for this project was written, the support of the Simons Foundation (\#284066, Jason Cantarella), and the support of Amazon Web Services under the AWS Educate program. We would also like to acknowledge many helpful conversations with our friends and colleagues, including Clayton Shonkwiler, Ken Millett, Chris Soteros, Stu Whittington, Gary Iliev, Tetsuo Deguchi, Rafa\l{} Komendarczyk, Andrew Rechnitzer, Julien Courtiel, Marni Mishna, and \'Eric Fusy.

\appendix
\section{pd-codes and isomorphisms}
\label{app:pdcodes}

In this appendix, we give more detail on our algorithm for detecting pd-isomorphic shadows. It is clear that a lot of data about a \pdcode is preserved by isomorphism: for instance, the number of crossings, edges, faces, and the numbers of edges around faces. We can use this information to rule out isomorphisms using a hashing scheme.

\begin{definition}
Suppose the \pdcode $P$ has $V$ crossings, $E$ edges, $F$ faces, and $C$ components. We assume that the faces are denoted $f_1, \dots, f_F$ and the components are denoted $c_1, \dots, c_C$. Further, let $\edges(x)$ give the number of edges on a face or component. Then the \emph{hash} of $P$ is given by the tuple
\begin{equation*}
\mathcal{H}(P) = (V,E,F,C,\{ \edges(f_1), \dots, \edges(f_F) \},
 \{ \edges(c_1), \dots, \edges(c_C) \}).
\end{equation*}
The last two are unordered sets of integers.
\end{definition}

It is clear that
\begin{lemma}
If two \pdcodes $P_1$ and $P_2$ are isomorphic, then $\mathcal{H}(P_1) = \mathcal{H}(P_2)$.
\end{lemma}

\begin{proof}
The numbers $V$, $E$, $F$, and $C$ are clearly preserved by isomorphism. The indices of edges and faces may be permuted by an isomorphism, but the number of edges on each can't change. Thus the \emph{unordered sets} of edge counts for faces and component remain the same as well.
\end{proof}

We can now build up an isomorphism between pdcodes by a series of definitions:

\begin{definition}
Suppose we have two pd-codes $\mathfrak{L}$ and $\mathfrak{L}'$ with the same hash.
\begin{itemize}
\item A bijection $\gamma: \{c_1,\dots,c_C\} \rightarrow \{c'_1,\dots,c'_C\}$ between the components of $\mathfrak{L}$ and the components of $\mathfrak{L}'$ is called \emph{component-length preserving} if $\# \edges(c_i) = \# \edges(\gamma(c_i))$ for all $i$.
\item Given such a component-length preserving bijection $\gamma$, a bijection $\epsilon : \{e_1,\dots,e_E\} \rightarrow \{e'_1,\dots,e'_E\}$ between the edges of $\mathfrak{L}$ and the edges of $\mathfrak{L}'$ is called \emph{component-preserving} and \emph{compatible with $\gamma$} if $\epsilon$ maps the edges of each $c_i$ to the edges of $\gamma(c_i)$ by an element of the dihedral group $D_{\edges(c_i)}$. That is, the edges of $c_i$ are mapped in cyclic (or reverse-cyclic) order to the corresponding edges of $\gamma(c_i)$.
\item Given a component-preserving bijection $\epsilon : \{e_1,\dots,e_E\} \rightarrow \{e'_1,\dots,e'_E\}$ between the edges of $\mathfrak{L}$ and the edges of $\mathfrak{L}'$, we say that a bijection $\nu : \{v_1, \dots, v_V\} \rightarrow \{v'_1,\dots,v'_V\}$ between the vertices of $\mathfrak{L}$ and the vertices of $\mathfrak{L}'$ is \emph{compatible with $\epsilon$} if
\begin{equation*}
\nu(\head(e_i)) = \head(\epsilon(e_i)) \quad \text{ and } \quad \nu(\tail(e_i)) = \tail(\epsilon(e_i))
\end{equation*}
when $e_i$ is part of a component mapped by an orientation-preserving element of the dihedral group, and
\begin{equation*}
\nu(\head(e_i)) = \tail(\epsilon(e_i)) \quad \text{ and } \quad \nu(\tail(e_i)) = \head(\epsilon(e_i))
\end{equation*}
when $e_i$ is part of a component mapped by an orientation-reversing element of the dihedral group.
\item Given $\gamma$, $\epsilon$, and $\nu$ that obey all the above conditions, we say that they are:
\begin{itemize}
 \item \emph{globally orientation-preserving} if the set of edges $e_i, e_j, e_k, e_l$ incident to each vertex $v$ of $\mathfrak{L}$ (in counterclockwise cyclic order) is mapped to the set of edges $\epsilon(e_i), \epsilon(e_j), \epsilon(e_k), \epsilon(e_l)$ incident to $\nu(v)$ in counterclockwise cyclic order.
 \item \emph{globally orientation reversing} if the $\epsilon(e_i), \epsilon(e_j), \epsilon(e_k), \epsilon(e_l)$ are incident to $\nu(v)$ but in clockwise cyclic order (for each $v$),
 \item otherwise, the triple is \emph{inconsistent}.
\end{itemize}
\end{itemize}
\label{def:compatible}
\end{definition}

We then have
\begin{proposition}
Given a pair of pd-codes $\mathfrak{L}$ and $\mathfrak{L}'$ with the same hash, each isomorphism of $\mathfrak{L}$ to $\mathfrak{L}'$ is given by a set of bijections $\gamma$ between their components, $\epsilon$ between their edges, and $\nu$ between their vertices where $\gamma$ is component-length preserving, $\epsilon$ is component-preserving and compatible with $\gamma$, $\nu$ is compatible with $\epsilon$ and the triple is globally orientation preserving or reversing (not inconsistent).
\end{proposition}

\begin{proof}
This is the definition of pd-isomorphic, restated using Definition~\ref{def:compatible}.
\end{proof}

A few other observations are helpful:
\begin{lemma}
If $\mathcal{H}(\mathfrak{L}) = \mathcal{H}(\mathfrak{L}')$ for pdcodes $\mathfrak{L}$ and $\mathfrak{L}'$, there is at least one component-length preserving $\gamma : \{c_1,\dots,c_C\} \rightarrow \{c'_1,\dots,c'_C\}$. All component-length preserving $\gamma$ can be generated by iterating over a product of permutation groups.
\end{lemma}


\begin{lemma}
Given a component-length preserving $\gamma$ and component-preserving and compatible $\epsilon$, and a set of orientations for the components of $\mathfrak{L}$, there is at most one $\nu : \{v_1, \dots, v_V\} \rightarrow \{v'_1,\dots,v'_V\}$ which is compatible with $\epsilon$ and consistently oriented on components and we can construct $\nu$ as below.
\end{lemma}

\begin{proof}
Each vertex $v$ of $\mathfrak{L}$ is incident to four edges $e_i, e_j, e_k, e_l$. Without loss of generality, assume that $v = \tail(e_i)$, $\tail(e_j)$, $\head(e_k)$ and $\head(e_l)$. Then if $\nu$ is compatible with $\epsilon$, we must have
\begin{equation*}
\nu(v) = \tail(\epsilon(e_i)) = \tail(\epsilon(e_j)) = \head(\epsilon(e_k)) = \head(\epsilon(e_l)).
\end{equation*}
If the four terms on the right are equal, this defines $\nu(v)$. If not, there is no compatible $\nu$.
\end{proof}

We can now find all isomorphisms between two pdcodes computationally by a simple brute-force strategy:

\begin{algorithmic}
\Procedure{BuildIsomorphisms}{$\mathfrak{L}$,$\mathfrak{L}'$}
\Comment{Build isomorphisms between pdcodes $\mathfrak{L}$ and $\mathfrak{L}'$}

\If {the hashes $\mathcal{H}(\mathfrak{L})$ and $\mathcal{H}(\mathfrak{L}')$ are different}
	\State $P$ and $P'$ are not isomorphic. Return $\emptyset$.
\EndIf

\ForAll{component-length preserving $\gamma: \{c_1,\dots, c_C\} \rightarrow \{c'_1, \dots, c'_C\}$}
	\ForAll{compatible and component-preserving $\epsilon$}
		\If{a compatible $\nu$ exists}
		\If{$\nu$ is globally orientation preserving or reversing}
			\State $\epsilon$, $\nu$ define an isomorphism $\mathfrak{L} \rightarrow \mathfrak{L}'$
		\EndIf
		\EndIf
	\EndFor
\EndFor
\EndProcedure
\end{algorithmic}



\section{Proving Proposition~\ref{prop:reduce}}
\label{app:reduction proof}

In this section, we give the proof of

\setcounter{theorem}{16}

\begin{proposition}
Every link shadow $L$ can be obtained from a connected, embedded planar simple graph of vertex degree $\leq 4$ $G_0$ by a series of $\loopinsert$, $\edgedouble$, $\cutedgedouble$, and $\pairinsert$ expansions.

Equivalently, any link shadow $L$ can be reduced to a connected embedded planar simple graph $L_0$ of vertex degree $\leq 4$ by a series of $\loopinsert$, $\edgedouble$, $\cutedgedouble$, and $\pairinsert$ reductions. The embedded isomorphism type of $L_0$ is determined uniquely by the (unoriented) shadow isomorphism type of $L$ (the order in which the reductions are performed doesn't matter).
\end{proposition}

\begin{proof}
We will prove the second statement, reducing in stages from some $G_n = G$ to $G_0$ by performing one reduction at each step. The number of steps we can perform is clearly finite, since each reduces the number of edges by at least one. So suppose we are at stage $G_i$. If there are no loop or multiple edges, we're done, and this is the simple graph $G_0$.

If there is a loop edge, we can remove it with a $\loopinsert$ move.

If there is a multiple edge, we must consider several cases. We can think of each vertex of $G_i$ as retaining a list of 4 connection points, ordered counterclockwise, from the initial embedding of $G$. Since we have performed some reductions already, some of these may be empty, but at least two are filled at each end of the multiple edge. Pick one vertex of the multiple edge and call it $v$ and the other vertex $w$.

If the edge multiplicity is four, $G$ is \hopfgraph. This is obtained from the graph with one edge and two vertices by three $\edgedouble$ moves.

If the edge multiplicity is three or two, there is at least one connection point on $v$ which is not occupied by a copy of the multiple edge followed immediately by a connection point which is occupied by a copy $e$ of the multiple edge. Without loss of generality, we'll call $e$ the \emph{base copy} of the multiple edge, and its connection point to at $v$ position $0$ around $v$. The remaining connection points will be numbered $1$, $2$, and $3$. By construction, the edge joined to $v$ at position $3$ (if any) is not connected to $w$. We can label the other end of the base copy $e$ position $a$ on the second vertex $w$, and label the other positions $b$, $c$, and $d$, again counterclockwise.

If the edge multiplicity is three, only one of these positions is unoccupied by a copy of the multiple edge. Looking at the three cases (shown in Figure \ref{fig:MultThree}), we can see that by parity, it must be position $b$, and the pair of copies $0a$ and $2c$ of the multiple edge can be removed by a $\pairinsert$ operation. We have now disposed of the case where edge multiplicity is three.
\begin{figure}[H]
\begin{center}
\includegraphics[width=4in]{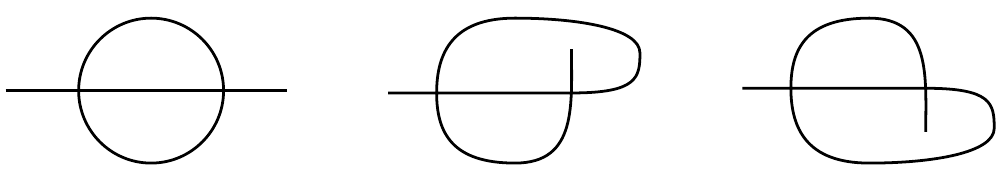}
\end{center}
\caption{By parity, because we came from a 4-regular embedded planar graph only the leftmost case can occur at any stage in the reduction process.}
\label{fig:MultThree}
\end{figure}

If edge multiplicity is two, there is one edge unaccounted for, which joins either position $1$ or $2$ on vertex $v$ to position $b$, $c$, or $d$ on vertex $w$. Therefore, there are six cases to address. We consider them in order, starting with the $1x$ configurations.
\begin{itemize}
\item
\begin{tabular}{m{1in}m{3in}m{1in}}
\begin{overpic}[width=1in]{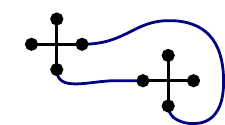}
	\put(23,52){\tiny{$2$}}
	\put(10,40){\tiny{$3$}}
	\put(35,40){\tiny{$1$}}
	\put(17,17){\tiny{$0$}}

	\put(73,36){\tiny{$d$}}
	\put(62,24){\tiny{$a$}}
	\put(85,24){\tiny{$c$}}
	\put(67,4){\tiny{$b$}}
\end{overpic}
&
In the $1b$ configuration, the multiple edge forms a 2-cycle dividing the portion of the graph $G$ connected to $cd$ from the portion connected to $23$. Deleting $1b$ requires a $\cutedgedouble$ move, and the remaining base edge is a cut edge of all further-reduced $G_i$, as shown at right.
&
\begin{overpic}[width=1in]{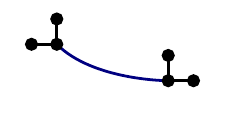}
	\put(23,52){\tiny{$2$}}
	\put(11.5,40.5){\tiny{$3$}}

	\put(73,35){\tiny{$d$}}
	\put(84.5,24){\tiny{$c$}}
\end{overpic}
\end{tabular}
\item
\begin{tabular}{m{1in}m{3in}}
\begin{overpic}[width=1in]{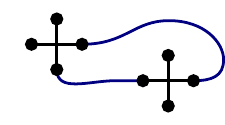}
	\put(23,52){\tiny{$2$}}
	\put(10,40){\tiny{$3$}}
	\put(35,40){\tiny{$1$}}
	\put(17,17){\tiny{$0$}}

	\put(73,36){\tiny{$d$}}
	\put(62,24){\tiny{$a$}}
	\put(85,24){\tiny{$c$}}
	\put(67,4){\tiny{$b$}}
\end{overpic}
&
The $1c$ configuration is forbidden by parity.
\end{tabular}
\item
\begin{tabular}{m{1in}m{3in}m{1in}}
\begin{overpic}[width=0.9in]{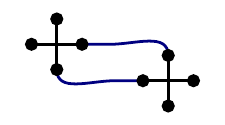}
	\put(23,52){\tiny{$2$}}
	\put(10,40){\tiny{$3$}}
	\put(35,40){\tiny{$1$}}
	\put(17,17){\tiny{$0$}}

	\put(73.5,36){\tiny{$d$}}
	\put(62,24){\tiny{$a$}}
	\put(85,24){\tiny{$c$}}
	\put(66,3){\tiny{$b$}}
\end{overpic}
&
In the $1d$ configuration, the multiple edge forms a bigon face. Deleting $1d$ uses an $\edgedouble$ reduction, and yields the configuration at right. The remaining base edge may or may not be a cut edge of the further $G_i$.
&
\begin{overpic}[width=0.9in]{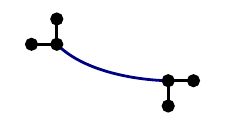}
	\put(23,52){\tiny{$2$}}
	\put(11.5,40.5){\tiny{$3$}}

	\put(73,-2){\tiny{$b$}}
	\put(84.5,24){\tiny{$c$}}
\end{overpic}
\end{tabular}
\end{itemize}
One might think that the $2-$ configurations are simply rearrangements of those above, but this is not true. A genuinely new case arises for $2c$.
\begin{itemize}
\item
\begin{tabular}{m{1in}m{3in}}
\begin{overpic}[width=0.9in]{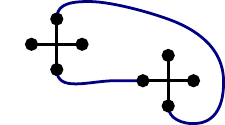}
	\put(20,52){\tiny{$2$}}
	\put(10,40){\tiny{$3$}}
	\put(35,40){\tiny{$1$}}
	\put(17,17){\tiny{$0$}}

	\put(73,36){\tiny{$d$}}
	\put(62,24){\tiny{$a$}}
	\put(85,24){\tiny{$c$}}
	\put(66,3){\tiny{$b$}}
\end{overpic}
&
The $2b$ configuration is forbidden by parity.
\end{tabular}
\item
\begin{tabular}{m{1in}m{3in}m{1in}}
\begin{overpic}[width=.9in]{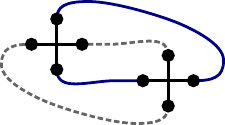}
	\put(20,52){\tiny{$2$}}
	\put(10,40){\tiny{$3$}}
	\put(35,40){\tiny{$1$}}
	\put(17,17){\tiny{$0$}}

	\put(74,36){\tiny{$d$}}
	\put(62,24){\tiny{$a$}}
	\put(85,24){\tiny{$c$}}
	\put(66,3){\tiny{$b$}}
\end{overpic}
&
In the $2c$ configuration, by parity, the graph $G$ must have connected $1$ and $d$ and also $3$ and $b$. None of our moves change the connectivity of the graph (because we never delete all copies of a multiple edge), so the current graph $G_i$ still joins these pairs of connection points. This means that we are in position for an $\pairinsert$ pair reduction, resulting in the graph at right.
&
\begin{overpic}[width=.9in]{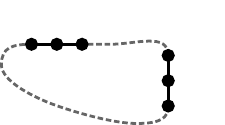}
	\put(11,40){\tiny{$3$}}
	\put(34,40){\tiny{$1$}}

	\put(79,30){\tiny{$d$}}
	\put(79,6){\tiny{$b$}}
\end{overpic}
\end{tabular}
\item
\begin{tabular}{m{1in}m{3in}}
\begin{overpic}[width=0.9in]{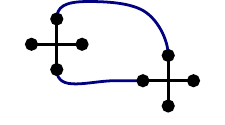}
	\put(20,52){\tiny{$2$}}
	\put(10,40){\tiny{$3$}}
	\put(35,40){\tiny{$1$}}
	\put(17,17){\tiny{$0$}}

	\put(77,36){\tiny{$d$}}
	\put(62,24){\tiny{$a$}}
	\put(85,24){\tiny{$c$}}
	\put(66,3){\tiny{$b$}}
\end{overpic}
&
The $2d$ configuration is forbidden by parity.
\end{tabular}
\end{itemize}
Along the way, our analysis has been entirely local: we need only consider a single vertex to decide whether we can apply an $\loopinsert$ reduction and a pair of vertices to decide on $\edgedouble$, $\cutedgedouble$, and $\pairinsert$ operations. To show that order of operations doesn't matter, we need to show that whether or not we can apply these operations does not depend on which reductions have already been performed. First, we note that since we never remove all copies of a multiple edge, we never change the connectivity of the graph during the reduction process.

The three copies of a multiplicity three edge must bound two bigons, and this does not change as we reduce other edges. Therefore, the $\pairinsert$ move is always available for all multiplicity three edges.

Whether a multiplicity two edge is eligible for an $\edgedouble$ move depends only on the positions of the ends of the multiple copies on their vertices, which don't change as we reduce. Therefore, this operation can always be performed (or is always forbidden), regardless of which reductions have already been performed.

Whether a multiplicity two edge is eligible for a $\cutedgedouble$ or $\pairinsert$ operation depends not only on the positions of ends of edges on their vertices, but also on the connectivity of the (reduced) graph. However, as we noted above, the connectivity of the graph doesn't change as we perform reductions.

It is clear that the isomorphism type of $G_0$ does not depend on the order of reduction-- after all, in the end we are simply reducing the multiplicity of multiple edges of the graph.

It takes only a moment longer to realize that the embedding of $G_0$ is determined as well-- this embedding is determined by the cyclic order of (surviving) edges around their vertices. We will have deleted some edges from many of these vertices by the time we reach $G_0$, potentially leaving many empty connections. However, the cyclic order of the surviving edges won't be affected by the order in which these connections were emptied.

One might worry that the choice of \emph{which}\footnote{Remember that the choice of ``base edge'' was arbitrary.} copy of an edge of multiplicity two to delete could affect the embedded isomorphism type after an $\edgedouble$ or $\cutedgedouble$ reduction, but it's easy to check that the two possible reduced configurations are (embedded) graph isomorphic by looking at the pictures above. Formally, the point is that the two copies of the edge are adjacent in the cyclic ordering of edges at each vertex, so the surviving copy is always in the same cyclic position relative to surviving edges incident to the vertex.
\end{proof}

\section{Branch and Bound Algorithm for Expansions}
\label{app:branch and bound}

Proposition~\ref{prop:reduce} suggests a strategy for generating diagrams: start by enumerating embedded planar simple graphs of vertex degree $\leq 4$ using \plantri, then expand them to $4$-regular embedded planar graphs using the moves above. We can generate embedded isomorphic graphs with different expansion sequences, so we will have to sort the graphs into isomorphism classes. We start with an embedded planar graph of vertex degree $\leq 4$.

\begin{lemma}
If $G_0$ is obtained from a $4$-regular embedded planar multigraph $G$ by the reduction process of Proposition~\ref{prop:reduce} then either every vertex of degree one in $G_0$ has exactly one loop edge in $G$ and one multiedge of multiplicity two obtained by $\edgedouble$ or $\cutedgedouble$ or the graph is \hopfgraph.
\label{lem:degreeone}
\end{lemma}

\begin{proof}
If we expand $G_0$ to $G$ using the four moves, three empty connections on the vertex must be filled during the process. If they are filled by redoubling the existing edge, then the degree of the vertex at the other end of the edge was also one, and we get \hopfgraph. Otherwise, we must fill two by adding a loop edge, and the other by doubling the existing edge.
\end{proof}

We also observe that two pairs of vertices $ab$ and $cd$ on the unit circle may be joined by nonintersecting chords inside the circle if and only if the pairs are unlinked on the circle (\raisebox{-0.032in}{\includegraphics[width=0.15in]{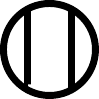}}
instead of \raisebox{-0.032in}{\includegraphics[width=0.15in]{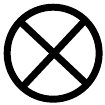}}). This happens when the order of the vertices around the circle is $abcd$ instead of $acbd$ or $adbc$.

We can now design an algorithm to produce all possible expansions of $G_0$, a given connected embedded planar simple graph of vertex degree $\leq 4$ as an integer constraint satisfaction problem. By Lemma~\ref{lem:degreeone}, we must add a loop to each vertex of degree one in $G_0$ eventually. We can save time by doing so at the start of the computation. We will therefore assume that loops have been added to create a \emph{prepared} graph $G_1$, and each vertex has degree $2$, $3$, or $4$.

We will now define four classes of variables:
\begin{itemize}
\item $l_{i}$ for every vertex $v_i$ of degree $2$
\item $d_{i,j}$ for every non-cut edge $e_{ij}$ in the graph joining vertices of degree $<4$.
\item $c_{i,j}$ for every cut edge $e_{ij}$ joining vertices of degree $<4$.
\item $p_{i,j}$ for every pair of vertices $v_{i}$, $v_j$ which both have degree 2 and are both on two different faces of the embedding
\end{itemize}
We take the subscripts to be unordered. That is, $d_{4,17}$ and $d_{17,4}$ are the same variable, since the edges $e_{4,17}$ and $e_{17,4}$ are the same edge.

These variables will all take the values $1$ or $0$, which represent the presence or absence of $\loopinsert$ loop edges, $\edgedouble$ or $\cutedgedouble$ doubles of existing edges, and $\pairinsert$ insertions of new pairs of edges.
We can now define two sets of equations relating these variables.

\begin{definition}
We define the \emph{vertex degree equations} for a prepared graph $G_1$ to be the collection of equations indexed by the vertices of $G_1$ given below. For each vertex index $i$ of degree $\delta(i)$
\begin{equation*}
\delta(i) + 2 l_i + \sum_j d_{i,j} + \sum_j c_{i,j} + 2 \sum_j p_{i,j} = 4
\end{equation*}
where the sums are taken over all $j$ for which the appropriate variables exist. These equations express the fact that in a complete expansion, the vertex degrees must all be four.
\end{definition}

The pair variables $p_{i,j}$ satisfy an additional set of equations:
\begin{definition}
For each $p_{i,j}$ and $p_{k,l}$ so that the vertices $v_i, v_j, v_k$ and $v_l$ are on the same pairs of faces, and so that the vertices are in the (cyclic) order $v_i, v_k, v_j, v_l$ or $v_i, v_l, v_j, v_k$ we have an additional \emph{linking equation}
\begin{equation*}
p_{i,j} + p_{k,l} \leq 1
\end{equation*}
\end{definition}

These equations express the fact that the edges corresponding to a linked pair of endpoints along a face must intersect inside the face. Therefore, if two pair variables are linked, at most one of them can take the value $1$. For instance, in the situation shown in Figure~\ref{fig:pair linking equation} there are four vertices of degree two along a pair of faces, we have six pair variables, two of which obey an additional linking equation.

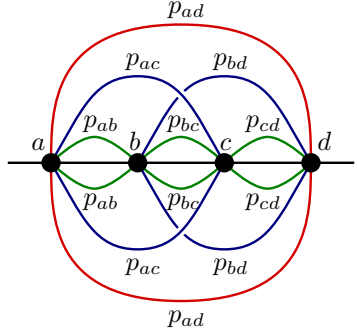
\begin{figure}[H]
\begin{center}
	\begin{overpic}[height=1.7in]{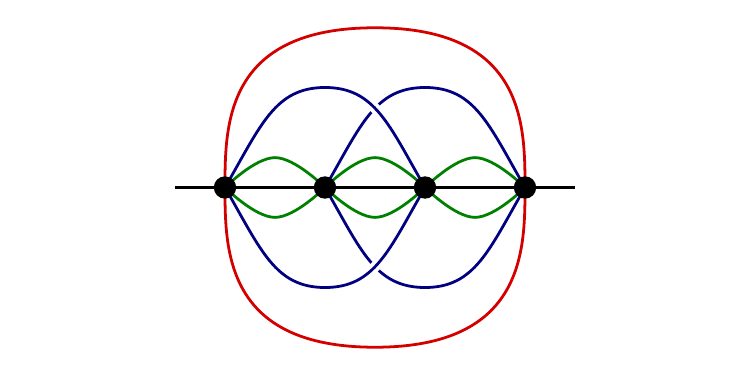}
	\put(27,27){$a$}
	\put(42,27){$b$}
	\put(56,27){$c$}
	\put(71,27){$d$}

	\put(48,48){{$p_{ad}$}}
	\put(48,0.5){{$p_{ad}$}}

	\put(41.5,40){{$p_{ac}$}}
	\put(41.5,8.5){{$p_{ac}$}}

	\put(55,40){{$p_{bd}$}}
	\put(55,8.5){{$p_{bd}$}}

	\put(35,30.5){{$p_{ab}$}}
	\put(35,18.5){{$p_{ab}$}}

	\put(47.9,30.5){{$p_{bc}$}}
	\put(47.9,18.5){{$p_{bc}$}}

	\put(60,30.5){{$p_{cd}$}}
	\put(60,18.5){{$p_{cd}$}}
	\end{overpic}
\end{center}
\caption{Suppose that the four vertices $a$, $b$, $c$ and $d$ are on a pair of different faces of the graph. There are six potential edges connecting these vertices. The variables $p_{ab}$, $p_{bc}$, $p_{cd}$ and $p_{da}$ obey only vertex degree equations. But the variables $p_{ac}$ and $p_{bd}$ obey both vertex degree equations and the additional pair linking equation $p_{ac} + p_{bd} \leq 1$.}
\label{fig:pair linking equation}
\end{figure}

\begin{proposition}
Every assignment of $\{0,1\}$ to the variables $l_i$, $d_{i,j}$, $c_{i,j}$, and $p_{i,j}$ which obeys the vertex degree equations and linking equations corresponds to an expansion of the connected planar graph $G_1$ (with vertex degrees $2$, $3$, and $4$ and loop edges only) to a collection of embeddings for the connected planar $4$-regular multigraph $G$.
\end{proposition}

\begin{proof}
Actually, there is only a little to check. By the arguments in the proof of Proposition~\ref{prop:reduce}, the order of expansion moves is irrelevant. So suppose there are $n$ moves, and fix an order for them. We must show that we can generate a family of graphs $G_1, G_2, \dots, G_n = G$. If we can perform the indicated expansions at all, we will generate a unique connected $4$-regular planar multigraph $G$ (we will see that the embedding of $G$ depends on choices we make along the way).  So suppose we have generated a given (embedded) $G_i$, and are trying to expand to $G_{i+1}$.

If the next expansion is an $\loopinsert$ expansion indicated by a positive $l_i$, it is possible as long as the vertex degree at $v_i$ is small enough. This is true because the corresponding vertex degree equation is satisfied by hypothesis.
We must choose which side of the edge to insert the loop; each choice yields a different embedding of $G_{i+1}$, and following the various possibilities will lead to a family of embeddings for $G_n = G$.

If the next expansion is an $\edgedouble$ indicated by a positive $d_{i,j}$, it is possible as long as the vertex degrees of $v_i$ and $v_j$ are small enough. This is true by their vertex degree equations. There is only one way to make this expansion, leading to a unique embedding for $G_{i+1}$.

If the next expansion is an $\edgedouble$ or $\cutedgedouble$ expansion indicated by a positive $c_{i,j}$ variable, it is (again) possible if the vertex degrees at $v_i$ and $v_j$ are small enough (which is again true by the vertex degree equations) \emph{and} if $e_{i,j}$ is a cut edge of $G_i$. We never apply these expansions more than once to an edge, so $e_{i,j}$ is a cut edge of $G_i$ since it was a cut edge of $G_1$. Choosing between $\edgedouble$ and $\cutedgedouble$ expansions will yield different embeddings of $G_{i+1}$ and we must follow both possibilities to generate the final family of embeddings of $G$.

This much was easy. If the next expansion is of type $\pairinsert$, there is more to check. First, we note that there is no ambiguity in embeddings here: if we can do the $\pairinsert$ expansion, we can do it in only one way and we generate a unique embedding of the graph $G_{i+1}$. But can we do it at all? Each $\pairinsert$ indicated by a positive $p_{i,j}$ requires several conditions. First, vertex degrees at $v_i$, $v_j$ must be small enough, which is true as usual because the vertex degree equations are satisfied.

We last have only to observe that by the vertex degree equations, the final graph $G$ is a $4$-regular planar multigraph. Since we have only added edges along the way, $G$ is connected because $G_1$ was.
\end{proof}

We have reduced the problem to that of building and satisfying the vertex degree and linking equations. This problem is basically standard, and we use the usual branch-and-bound algorithm. We must define a canonical order on the variables (it doesn't matter how, but to be specific, in our implementation we sort the classes of variables in the order $l_i \prec d_{i,j} \prec c_{i,j} \prec p_{ij}$ and in dictionary order by the (sorted) pair $\{i,j\}$ within each class). Then we enumerate the possible assignments of $\{0,1\}$ to variables recursively, pruning the tree whenever a vertex degree or linking equation is violated. As usual, this is in theory possibly exponentially slow, but in practice efficient enough for small $n$.

We now consider the problem of dividing the results into embedded isomorphism classes. We first observe that we have already shown in Proposition~\ref{prop:reduce} two different reduced graphs $G_0$ and $G_0'$ cannot expand to the same $G$ since the embedded isomorphism type of the reduction $G_0$ is determined by the embedded isomorphism type of the expansion. However, it is possible for two different collections of expansion moves for the \emph{same} graph $G_0$ to produce isomorphic $G$ and $G'$ as in Figure~\ref{fig:expansion woes} above. Therefore, we insert the various expansions of each graph into a pdstor (see Definition~\ref{def:pdstor}) with isomorphism checking and then combine the contents of these pdstors into a final list of shadows.

%
%
%
%
%
%

\bibliography{knotprobabilitypapers,knotprob_extra}

\end{document}